\documentclass[11pt,reqno]{amsart}

\usepackage[T1]{fontenc}

\usepackage{amsmath,xypic}								
\usepackage{amssymb}
\usepackage{amsthm}
\usepackage{amscd}
\usepackage{amsfonts}
\usepackage{stmaryrd}
\usepackage[all]{xy}

\usepackage{caption}
\usepackage{subcaption}
\usepackage{euler}

\usepackage{extarrows}

\usepackage[colorlinks, linktocpage, citecolor = purple, linkcolor = blue]{hyperref}
\usepackage{color}

\usepackage{tikz}									
\usetikzlibrary{matrix}
\usetikzlibrary{patterns}
\usetikzlibrary{matrix}
\usetikzlibrary{positioning}
\usetikzlibrary{decorations.pathmorphing}
\usetikzlibrary{cd}

\usepackage{fullpage}
\usepackage[shortlabels]{enumitem}

\numberwithin{equation}{section}

\newtheorem{maintheorem}{Theorem}	
\newtheorem{maincorollary}[maintheorem]{Corollary}								
\newtheorem{theorem}{Theorem}[section]
\newtheorem{lemma}[theorem]{Lemma}
\newtheorem{proposition}[theorem]{Proposition}

\theoremstyle{definition}
								
\newtheorem{definition}[theorem]{Definition}
\newtheorem{question}[theorem]{Question}
\newtheorem{example}[theorem]{Example}

\newtheorem{notation}[theorem]{Notation}

\newtheorem{remark}[theorem]{Remark}
\newtheorem{remarks}[theorem]{Remarks}

\newcommand{\hooklongrightarrow}{\lhook\joinrel\longrightarrow}

\newcommand{\R}{\mathbb{R}}

\newcommand{\Z}{\mathbb{Z}}

\newcommand{\N}{\mathbb{N}}
\newcommand{\PP}{\mathbb{P}}
\newcommand{\NN}{\mathbb{N}}

\newcommand{\G}{\mathbb{G}}

\newcommand{\TT}{\mathbb{T}}

\newcommand{\calA}{\mathcal{A}}

\newcommand{\calD}{\mathcal{D}}

\newcommand{\calM}{\mathcal{M}}
\newcommand{\calO}{\mathcal{O}}

\newcommand{\calU}{\mathcal{U}}
\newcommand{\calV}{\mathcal{V}}

\newcommand{\calX}{\mathcal{X}}

\newcommand{\calZ}{\mathcal{Z}}
\newcommand{\calMbar}{\overline{\mathcal{M}}}

\DeclareMathOperator{\Pic}{Pic}
\DeclareMathOperator{\Spec}{Spec}
\DeclareMathOperator{\Hom}{Hom}

\DeclareMathOperator{\Aut}{Aut}

\DeclareMathOperator{\trop}{trop}
\DeclareMathOperator{\val}{val}

\DeclareMathOperator{\Trop}{Trop}

\DeclareMathOperator{\Div}{Div}

\DeclareMathOperator{\calDiv}{\mathcal{D}iv}
\DeclareMathOperator{\calDivbar}{\overline{\mathcal{D}iv}}
\DeclareMathOperator{\Gr}{Gr}
\DeclareMathOperator{\supp}{supp}
\DeclareMathOperator{\PDiv}{PDiv}
\DeclareMathOperator{\mdeg}{mdeg}
\DeclareMathOperator{\ev}{ev}

\let\Rat\relax
\DeclareMathOperator{\Rat}{Rat}
\let\div\relax
\DeclareMathOperator{\div}{div}
\let\ord\relax
\DeclareMathOperator{\ord}{ord}

\title[Symmetric powers of algebraic and tropical curves]{Symmetric powers of algebraic and tropical curves: A~non-Archimedean perspective}

\author{Madeline Brandt}
\address{Department of Mathematics, University of California, Berkeley, 970 Evans Hall, Berkeley, CA 94720}
\email{\href{mailto:brandtm@berkeley.edu}{brandtm@berkeley.edu}}

\author{Martin Ulirsch}
\address{Institut f\"ur Mathematik, Goethe-Universit\"at Frankfurt, 60325 Frankfurt am Main, Germany, and \\
Mathematics Institute, University of Warwick, Coventry CV4 7AL, United Kingdom}
\email{\href{mailto:ulirsch@math.uni-frankfurt.de}{ulirsch@math.uni-frankfurt.de}}


\begin{document}

\begin{abstract} 
We show that the non-Archimedean skeleton of the $d$-th symmetric power of a smooth projective algebraic curve $X$ is naturally isomorphic to the $d$-th symmetric power of the tropical curve that arises as the non-Archimedean skeleton of $X$. The retraction to the skeleton is precisely the specialization map for divisors. Moreover, we show that the process of tropicalization naturally commutes with the diagonal morphisms and the Abel-Jacobi map and we exhibit a faithful tropicalization for symmetric powers of curves. Finally, we prove a version of the Bieri-Groves Theorem that allows us, under certain tropical genericity assumptions, to deduce a new tropical Riemann-Roch-Theorem for the tropicalization of linear systems. 
\end{abstract}

\maketitle

\setcounter{tocdepth}{1}
\tableofcontents


\section*{Introduction}

Throughout, let $K$ be a non-Archimedean field with valuation ring $R$ whose residue field $k$ is algebraically closed and contained in $K$. Let $X$ be a smooth projective curve over $K$ of genus $g$ and let $d\geq 0$. The \emph{$d$-th symmetric power $X_d$} of $X$ is defined to be the quotient 
\begin{equation*}
    X_d=X^d/S_d
\end{equation*}
of the $d$-fold product $X^d=X\times \cdots \times X$ by the action of the symmetric group $S_d$ that permutes the entries. The symmetric power $X_d$ is again a smooth and projective algebraic variety and functions as the fine moduli space of effective divisors of degree $d$ on $X$ (see \cite[Section 3]{Milne_Jacobian} for details).

Let $\Gamma=\Gamma_X$ be the dual tropical curve of $X$, i.e. the minimal skeleton of $X^{\text{an}}$. As a set, the \emph{$d$-th symmetric power} $\Gamma_d$ of $\Gamma$ is defined to be the quotient
\begin{equation*}
    \Gamma_d=\Gamma^d/S_d
\end{equation*}
of the $d$-fold product by $S_d$-action. We will see in Section~\ref{section_sympow} that, once we choose a semistable model $(G,\vert.\vert)$ for $\Gamma$, the symmetric power $\Gamma_d$ naturally carries the structure of a \emph{colored polysimplicial complex} and it naturally functions as a moduli space of effective divisors of degree $d$ on $\Gamma$.

Let $\calX$ be a semistable model of $X$ over $R$ that admits a section. The special fiber of $\calX$ is a semistable curve whose weighted dual graph (together with the edge lengths given by the valuations of the deformation parameters at every node) provides us with a natural choice of a model $\big(G,\vert.\vert\big)$ of $\Gamma$. There is a natural tropicalization map 
\begin{equation*}
    \trop_{X_d}\colon X_d^{an}\longrightarrow \Gamma_d
\end{equation*}
given by pushing forward an effective Cartier divisor $D$ on $X_L$, for a non-Archimedean extension $L$ of $K$, to the dual tropical curve $\Gamma_{X_L}=\Gamma_X$, which is essentially a version of Baker's specialization map for divisors in \cite{Baker_specialization} (see Section~\ref{section_specialization} for details). 

On the other hand, using a variation of the compactification of the moduli space of effective divisors over $\calM_g$ constructed in \cite[Section 2]{MoellerUlirschWerner_realizability}, a special case of the moduli space of stable quotients in   \cite{MarianOpreaPandharipande}, we find a polystable model $\overline{\Div}_d^+(\calX)$ of $X_d$ over $\Spec R$ that has a natural modular interpretation. The space $\overline{\Div}_d^+(\calX)$ is not the $S_d$-quotient of the fibered product $\calX\times_{\Spec R}\cdots\times_{\Spec R}\calX$ but rather a  resolution thereof with good moduli-theoretic properties. 

By \cite{Berkovich_analytic}, associated to the polystable model $\overline{\Div}_d^+(\calX)$ there is a strong deformation retraction 
\begin{equation*}
    \rho_{X_d}\colon X_d^{an}\longrightarrow \Sigma(X_d)
\end{equation*} 
onto the \emph{non-Archimedean skeleton} $\Sigma(X_d)$ of $X_d^{an}$, which naturally carries the structure of a colored polysimplicial complex. We refer the reader to Section~\ref{section_polystableskeletons} for a guide to this construction. Our main result is the following Theorem~\ref{thm_skeleton=symmetricpower}.

\begin{maintheorem}\label{thm_skeleton=symmetricpower}
Let $X$ be a smooth and projective algebraic curve over $K$. Let $\calX$ be a fixed semistable model of $X$ over the valuation ring $R$ of $K$ that admits a section.  Denote by $\Gamma$ the dual tropical curve of $X$. There is a natural isomorphism
\begin{equation*}
\mu_{X_d}\colon \Gamma_d\xlongrightarrow{\sim}\Sigma(X_d)
\end{equation*}
of colored polysimplicial complexes that makes the diagram
\begin{center}\begin{tikzcd}
 X_d^{an} \arrow[dr,"\rho_{X_d}"'] \arrow[drrr, bend left, "\trop_{X_d}"]& & &\\
 & \Sigma(X_d) & & \Gamma_d \arrow[ll, "\mu_{X_d}","\sim"']
\end{tikzcd}\end{center}
commute. 
\end{maintheorem}

In other words, the symmetric power $\Gamma_{d}$ of the skeleton $\Gamma=\Gamma_X$ of $X^{an}$ is isomorphic to the skeleton $\Sigma(X_d)$ of the symmetric power $X_d^{an}$ of $X^{an}$. The main idea of our proof is to carefully describe the combinatorial structure stratification of the polystable model of $X_d$ coming from \cite{MoellerUlirschWerner_realizability} and to identify it with the discrete data of $\Gamma_d$ (thought of as the space $\Div_d^+(\Gamma)$ of effective divisors on $\Gamma$). 

\subsection*{An alternative to Theorem \ref{thm_skeleton=symmetricpower}} A slightly different version of Theorem \ref{thm_skeleton=symmetricpower} has appeared in \cite{Shen_Lefschetz} en route to the proof of a non-Archimedean Lefschetz hyperplane theorem for the the locus of effective divisors in the Picard group. In \cite[Section 6]{Shen_Lefschetz} the author first identifies the non-Archimedean skeleton of the $d$-fold product $X^d$ with the $d$-fold product using Berkovich's skeleton construction in \cite{Berkovich_analytic} and then shows that the deformation retraction is naturally $S_n$-invariant. This implies that the skeleton of the quotient $X_d^{an}=(X^d)^{an}/S_n$ is equal to $\Gamma_d$ as a set. We refer the reader to \cite{BrownMazzon} for more details on skeletons associated to products of degenerations.
The resulting polyhedral structure on the skeleton constructed in \cite{Shen_Lefschetz}, however, is not the one we introduce in Section \ref{section_sympow} below.
  
  In this article we prefer our approach via the explicit model $\overline{\Div}_d^+(\calX)$, since its modular interpretation simplifies the construction of the tropicalization map and the combinatorial stratification of its special fiber "explains" where the a priori only intrinsically defined polyhedral structure on $\Gamma_d$ is coming from. 

\subsection*{Tropicalization of subvarieties} Let $Y\subseteq X_d$ be a closed subvariety. We define the \emph{tropicalization} $\Trop_{X_d}(Y)$ of $Y$ to be the closed subset of $\Gamma_d$ given by the projection of $Y^{an}$ to $\Gamma_d$ via $\trop_{X_d}$, i.e. essentially via the specialization of effective divisors from $X$ to $\Gamma_X$ from \cite{Baker_specialization}. In other words, we set 
\begin{equation*}
 \Trop_{X_d}(Y):=\trop_{X_d}(Y) \ .   
\end{equation*}
By Theorem~\ref{thm_skeleton=symmetricpower} this is nothing but the projection of $Y^{an}$ to the skeleton of $X_d^{an}$ via $\rho_{X_d}$.

A surprisingly useful consequence of Theorem~\ref{thm_skeleton=symmetricpower} is that the continuity of $\rho_{X_d}$ implies the continuity of the tropicalization $\trop_{X_d}$. This allows us to deduce a collection of functoriality results in Section~\ref{section_functoriality} from the linearity of Baker's specialization map in  \cite{Baker_specialization} and from the compatibility of the process of tropicalization with the Abel-Jacobi map proved in \cite[Theorem 1.3]{BakerRabinoff_skelJac=Jacskel}. Moreover, the usual arguments from the proof of  \cite[Theorem 2.2.7]{EinsiedlerKapranovLind} (also see \cite[Proposition 3.5]{Gubler_guide}) immediately imply Corollary~\ref{cor_connectedness}. 

\begin{maincorollary}\label{cor_connectedness}
If $Y\subseteq X_d$ is connected, then the tropicalization $\Trop_{X_d}(Y)$ is connected as well.
\end{maincorollary}
\begin{proof}
Since $Y$ is connected $Y^{an}$ is also connected \cite[Theorem 3.4.8]{Berkovich_book}, and then continuity of the tropicalization map implies that $\Trop_{X_d}(Y)$ is connected.
\end{proof}

Denote by $\Lambda\subseteq \R$ the value group $-\log\vert K^\ast\vert$ of $K$. In Section~\ref{section_BieriGroves} we prove a generalization of the classical Bieri-Groves-Theorem (see \cite[Theorem A]{BieriGroves} and \cite[Theorem 2.2.3]{EinsiedlerKapranovLind}) for projections to the skeleton associated to a polystable model, which for $X_d$ can be stated as follows. 

\begin{maintheorem}\label{thm_BieriGrovessympow}
Let $X$ be a smooth projective curve over $K$ and let $\calX$ be a fixed semistable model of $X$ over $R$. Suppose that $Y\subseteq X_d$ is a closed subscheme that is is equidimensional of dimension $\delta$ such that $Y\cap X_d^\circ\neq\emptyset$, where $X_d^\circ$ parametrizes the reduced divisors on $X$. Then the tropicalization 
\begin{equation*}
\Trop_{X_d}(Y) := \trop_{X_d}(Y^{an})
\end{equation*}
of $Y$ is a $\Lambda$-rational polyhedral complex of dimension at most $\delta$. If $X$ has bad reduction and the tropicalization contains a point in the interior of a maximal cell of $X_d^{trop}$, then the dimension of $\Trop_{X_d}(Y)$ is equal to $\delta$.
\end{maintheorem}

The theory developed in this article allows us to study the tropical geometry of linear series by directly tropicalizing them as subvarieties of $X_d$. For example, Theorem~\ref{thm_BieriGrovessympow} immediately implies the following \emph{realizable Riemann-Roch Theorem}. 

\begin{maincorollary}\label{cor_rRR}
Let $X$ be an algebraic curve with bad reduction of genus $g$ and let $D$ be a divisor on $X$ of degree $d$. Let $|D| \subset X_d$ denote the linear series of $D$. Suppose that both $\Trop_{X_d}\big\vert D\big\vert$ and $\Trop_{X_d}\big\vert K_X-D\big\vert$ contain a point in the interior of a maximal cell of $\Gamma_d$. Then we have:
\begin{equation*}
\dim \Trop_{X_d} \big\vert D\big\vert - \dim\Trop_{X_d}\big\vert K_X-D\big\vert = d-g+1 \ . 
\end{equation*}
\end{maincorollary}
 
By \cite{Baker_specialization}, the dimension of $\Trop_{X_d}\big\vert D\big\vert$ is not always equal to the rank of the specialization of $D$ to $\Gamma$. So, in particular, the realizable Riemann-Roch Theorem does, in general, not imply the well-known intrinsic tropical Riemann-Roch Theorem from \cite{BakerNorine, GathmannKerber, AminiCaporaso}. In the special case when $\Gamma$ is a generic chain of loops, however, the lifting results of \cite{CartwrightJensenPayne} allow us to say more.

\subsection*{Algebraic and tropical Riemann-Roch Theorem for generic chains of loops}\label{section_RiemannRoch}


 Let $\Gamma$ be a chain of $g$ loops, where each loop consists of two edges having lengths $l_i$ and $m_i$ (see Figure~\ref{fig:chainofloops}). Suppose that $\Gamma$ is a \emph{generic}, i.e. suppose that none of the ratios $l_i/m_i$ is equal to the ratio of two positive integers whose sum is less than or equal to $2g-2$ (see \cite[Definition 4.1]{CoolsDraismaPayneRobeva}). The results of \cite{CartwrightJensenPayne} show that the algebraic Riemann-Roch-Theorem implies the tropical Riemann-Roch-Theorem.
 
 \begin{figure}
     \centering
     \includegraphics[height = 1.7 in]{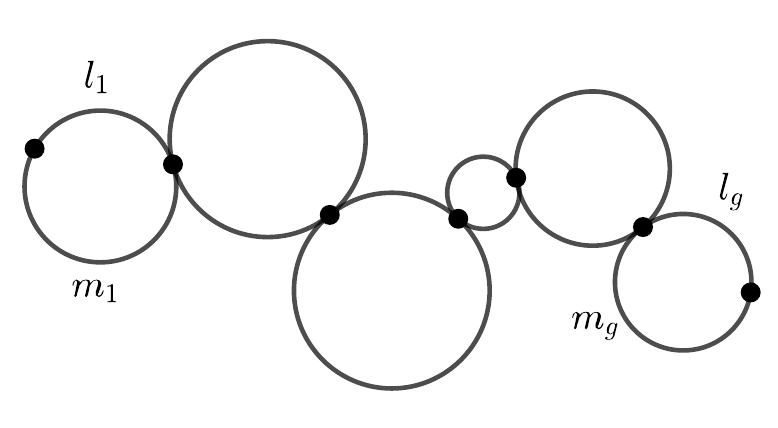}
     \caption{A chain of loops with edge lengths labeled.}
     \label{fig:chainofloops}
 \end{figure}

Let $D$ a divisor on $\Gamma$ of degree $d$ and rank $r$ supported on  $\Lambda$-rational points of $\Gamma$. Let $X$ be a Mumford curve whose dual tropical curve is $\Gamma$. By \cite{CartwrightJensenPayne} there is a line bundle $L$ of degree $d$ and rank $r$ on $X$ such that the specialization of $L$ is equal to the divisor class $[D]$. The construction in \cite{CartwrightJensenPayne} is naturally compatible with residue duality and thus the specialization of $\omega_X\otimes L^{-1}$ is equal to the class $[K_\Gamma-D]$ and the rank of $\omega_X\otimes L^{-1}$ is equal to the rank of $K_\Gamma-D$.  Therefore we have
\begin{equation*}
r(D)-r(K_\Gamma-D)=\dim\big\vert L\big\vert -\dim\big\vert \omega_X \otimes L^{-1}\big\vert =d-g+1  
\end{equation*}
and, in this situation, the algebraic Riemann-Roch Theorem implies its tropical counterpart. If both $\Trop_{X_d}\vert L\vert$ and $\Trop_{X_d}\vert L\otimes \omega_X\vert$ contain an interior point of a maximal cell of $\Gamma_d$, then we have $r(D)=\dim\Trop_{X_d}\vert L\vert$ as well as $r(K_\Gamma-D)=\dim\Trop_{X_d}\vert\omega_X\otimes L^{-1}\vert$ and the realizable Riemann-Roch-Theorem from above is equivalent to the tropical Riemann-Roch-Theorem. One may think of $\Trop_{X_d} \vert L\vert$ as the realizability locus in the tropical linear system $\vert D\vert$ and the Baker-Norine rank is equal to the polyhedral dimension of $\Trop_{X_d} \big\vert L\big\vert$.

\subsection*{Faithful tropicalization} The classical approach to the process of tropicalization goes by choosing an embedding into a suitable toric variety and then applying  coordinate-wise valuations to the embedded variety. For symmetric powers, however, Theorem~\ref{thm_skeleton=symmetricpower} suggests that it might be more natural to think of tropicalization as a projection to the non-Archimedean skeleton. The principle of \emph{faithful tropicalization}, as pioneered in \cite{BakerPayneRabinoff} and further developed e.g. in \cite{GublerRabinoffWerner, CuetoHaebichWerner, GublerRabinoffWernerII}, seeks to realign these two perspectives.

Expanding on \cite{GublerRabinoffWerner} we prove in Section~\ref{section_faithfultropicalization} a faithful tropicalization result for skeletons associated to polystable models. As a consequence we obtain the following Theorem~\ref{thm_faithfulsympow}.

\begin{maintheorem}\label{thm_faithfulsympow}
There is an open subset $U\subseteq X_d$ as well as a morphism $f\colon U\rightarrow \G_m^n$ such that the restriction  $\trop\circ f^{an}$ to  $\Gamma_d\simeq\Sigma(X_d)\subseteq U^{an}$  of the induced tropicalization map
\begin{equation*}
\trop_f\colon U^{an}\xrightarrow{f^{an}}\G_m^{n,an}\xrightarrow{\trop}\R^n
\end{equation*}
is a homeomorphism onto its image in $\R^n$. If $\mathcal{X}$ is a strictly semistable model for $X$, then the restriction of $\trop_f$ to each cell of $\Gamma_d$ is unimodular.  
\end{maintheorem}


Unfortunately the construction of the map $f\colon U\rightarrow \G_m^n$ is by no means effective. In Section~\ref{section_faithfultropvialinearseries} we speculate how the recent work of Kawaguchi and Yamaki \cite{KawaguchiYamaki} that uses linear series to find effective faithful tropicalizations of curves may be generalized to find effective faithful tropicalization of symmetric powers. 

Finally, in Section~\ref{section_deJonquieres} we describe a further open question, on the tropical geometry of varieties of de Jonquier\`es divisors in a fixed linear system (expanding on the recent work of Ungureanu \cite{Ungureanu_deJonquieres}). 

\subsection*{Complements and related works}
Symmetric powers of tropical curves have already appeared in \cite{MikhalkinZharkov, HaaseMusikerYu, GathmannKerber}, where they form a key ingredient to understand the polyhedral structure of tropical linear series. While the process of tropicalization on the level of divisor classes, e.g. as a tropicalization map of Picard groups, has been studied widely (e.g. in \cite{Baker_specialization, BakerRabinoff_skelJac=Jacskel}), the purpose of this paper is to provide non-Archimedean foundations for the tropicalization of symmetric powers. 

Our perspective on Theorem~\ref{thm_skeleton=symmetricpower} is that it is another incarnation of the principle that the non-Archimedean skeleton of an algebraic moduli space typically is a tropical moduli space that has recently been implemented in a multitude of different cases, e.g. in  \cite{BakerRabinoff_skelJac=Jacskel} for Picard varieties, in \cite{AbramovichCaporasoPayne} for the moduli space of curves, in \cite{CavalieriMarkwigRanganathan} for Hurwitz space, in \cite{CavalieriHampeMarkwigRanganathan, Ulirsch_tropHassett} for Hassett spaces, in \cite{Ranganathan_skeletonsofstablemaps, RanganathanSantosParkerWiseII} for rational and elliptic stable maps, and in \cite{AbreuPacini} for the universal Picard variety. 

In fact, in \cite[Theorem 3]{MoellerUlirschWerner_realizability}, M\"oller, the second author, and Werner prove that the skeleton of the moduli space $\calDiv_{g,d}$ of effective divisors over $\calM_g$ is equal to the moduli space of effective tropical divisors $\Div_{g,d}^{trop}$  over $M_g^{trop}$. It is tempting to speculate that this result would imply our Theorem~\ref{thm_skeleton=symmetricpower}. Unfortunately our current understanding of the functoriality of skeleton constructions does not seem to allow us to formally deduce such a result. The main obstacle to overcome here lies in the fact that the functor that associates to a $K$-analytic space its underlying topological space does not preserve fibered products.

\subsection*{Acknowledgements} The idea for this article was conceived while both authors were visiting the  Max-Planck-Instute for Mathematics in the Sciences in Leipzig in Summer 2017; they would like to thank their generous host Bernd Sturmfels, who also provided useful comments and feedback on an earlier version of this article. We thank Sam Payne for pointing us to the closely related results in article \cite{Shen_Lefschetz}. M.U. would like to thank Enrica Mazzon for explaining the techniques in her article \cite{BrownMazzon}. Thanks are also due to the anonymous referee for several helpful suggestions and comments.

 This project  has  received  funding  from  the  European Union's Horizon 2020 research and innovation programme  under the Marie-Sk\l odowska-Curie Grant Agreement No. 793039. \includegraphics[height=1.7ex]{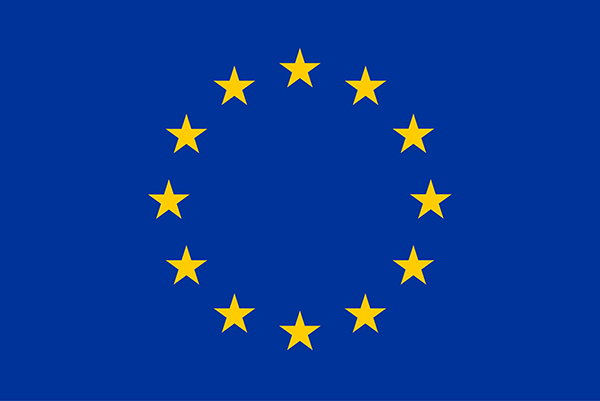} 
  We also acknowledge support from the LOEWE-Schwerpunkt ``Uniformisierte Strukturen in Arithmetik und Geometrie''.
 
 This material is based upon work supported by the National Science Foundation Graduate Research Fellowship Program under Grant No. DGE 1752814. Any opinions,
findings, and conclusions or recommendations expressed in this material are those of the
author(s) and do not necessarily reflect the views of the National Science Foundation.


\section{Symmetric powers of tropical curves}\label{section_sympow}

\subsection{Colored polysimplicial complexes}
In this section we build polyhedral complexes from simplices and products of simplices which have sizes, or \emph{colors}, and only glue along faces that have the same color. 

Let $n\in \N$. We identify the non-negative positive orthant $\mathbb{R}_{\geq 0}^{n+1}$ with the space of monoid homomorphisms $\Hom(\N^{n},\R_{\geq 0})$ into the additive monoid $(\mathbb{R}_{\geq 0},+)$. A \emph{colored $n$-simplex} is a topological space $\Delta$ together with an injective homomorphism $\phi_\Delta\colon\N^{n+1}\hookrightarrow C^0(\Delta,\R_{\geq 0})$ such that the induced map \begin{equation*}\begin{split}\Delta&\longrightarrow \R_{\geq 0}^{n+1}=\Hom(\N^{n+1},\R_{\geq 0})\\
x&\longrightarrow \big(m\mapsto\phi_\Delta(m)(x)\big)
\end{split}\end{equation*}
is a homeomorphism of $\Delta$ onto a subset
\begin{equation}\label{eq_coloredsimplex}
    \Delta(n,a):=\left\{(x_0,\ldots, x_n)\in\R_{\geq 0}^{n+1}\middle\vert \sum_{j=0}^n x_i =a\right\}\subseteq \R_{\geq 0}^{n+1}
\end{equation}
for some $a\in \R_{>0}$. The real number $a$ is uniquely determined by the datum $(\Delta, \phi_\Delta)$ and is called the \emph{color} of $\Delta$. 
More generally, we define the notion of a \emph{colored polysimplex}.

\begin{definition}
Let $\vec{n} = (n_1, \ldots, n_k) \in \mathbb{N}^k$ and $\vec{a} = (a_1, \ldots, a_k) \in \mathbb{R}_{>0}^k$ a vector of colors. Set $\vert \vec{n}\vert:=\sum_{i=1}^k(n_i+1)$. A \emph{colored $\vec{n}$-polysimplex} with \emph{colors} $\vec{a}$ is a topological space $\Delta$ together with an injective homomorphism $\phi_\Delta\colon \N^{\vert\vec{n}\vert}\hookrightarrow C^0(\Delta,\R)$ such that the induced map 
\begin{equation*}\begin{split}\Delta&\longrightarrow \R_{\geq 0}^{\vert\vec{n}\vert}=\Hom(\N^{\vert\vec{n}\vert},\R_{\geq 0})\\
x&\longrightarrow \big(m\mapsto\phi_\Delta(m)(x)\big)
\end{split}\end{equation*}
is a homoemorphism of $\Delta$ onto the subset
\begin{equation}
\label{eqn:deltanr}
    \Delta(\vec{n},\vec{a}):= \left \{ (x_{ij})_{1\leq i \leq k, 0 \leq j \leq n_i} \in \mathbb{R}^{n_1+1+\cdots+n_k+1}_{\geq 0} \middle | \sum_{j=0}^{n_i} x_{ij} = a_i \text{ for all }i \right \}.
\end{equation}
\end{definition}
  
We refer to $\Delta(n,a)$ as the standard colored simplex of type $(n,a)$ and to $\Delta(\vec{n},\vec{a})$ as the \emph{standard colored polysimplex} of type $(\vec{n},\vec{a})$.

\begin{example}
The poly\-simplex $\Delta((1,1),(2,3))$ is the set
$$
\left\{
(x_{10},x_{11},x_{20},x_{21}) \in \mathbb{R}^4_{\geq 0}
\middle |\ 
x_{10}+x_{11} = 2,\ x_{20}+x_{21} = 3
\right\}.
$$
This is a two dimensional rectangle inside of $\mathbb{R}^4$, with sides of length 2 and 3. 
\end{example}

We may think a colored polysimplex as a finite product of colored simplices. In fact, we naturally have 
\begin{equation*}
    \Delta(\vec{n},\vec{a}) =\Delta(n_1,a_1)\times\cdots\times\Delta(n_k,a_k) \ .
\end{equation*}
A \emph{morphism} $f\colon (\Delta,\phi_\Delta)\rightarrow (\Delta',\phi_{\Delta'})$ of colored polysimplices is a continuous map $f\colon\Delta\rightarrow \Delta'$ such that the pullback homomorphism 
\begin{equation*}\begin{split}
    C^0(\Delta',\mathbb{R})&\longrightarrow C^0(\Delta, \mathbb{R})\\
    u&\longrightarrow f^\ast u:=u\circ f
\end{split}\end{equation*}
restricts to a homomorphism of monoids  $\N^{\vert\vec{n}'\vert}\rightarrow \N^{\vert \vec{n}\vert}$. Here and in the rest of the paper we think of $\N^{\vert\vec{n}\vert}$ as an embedded submonoid of $C^0(\Delta,\R)$ and often drop the reference to $\phi_\Delta$ from our notation.

A \emph{face} of a colored polysimplex $(\Delta, \phi)$ is the zero set of a function $\phi(m)\in C^0(\Delta, \R)$ for an $m\in\N^{\vert \vec{n}\vert}$. Each face inherits the structure of a colored polysimplex by restriction of the functions in $\N^{\vert \vec{n}\vert}$. We say that a morphism $\Delta\rightarrow \Delta'$ is a \emph{face morphism}, if it induces an isomorphism of $\Delta$ with a face of $\Delta'$.


\begin{example}
Let $\Delta$ be the standard $((1,1),(2,3))$-poly\-simplex. It has a face $F$ given by
$$
\left\{
(x_{10},x_{11},x_{20},x_{21}) \in \mathbb{R}^4_{\geq 0}
\middle |\ 
x_{10} = 2,\ x_{20}+x_{21} = 3, x_{11} = 0
\right\}.
$$
After projection, this is the standard $((1),(3))$ polysimplex.
\end{example} 



We denote by $\mathbf{cPoly}$ the category of standard colored polysimplices with face morphisms. 
In the following Definition~\ref{polysimplex} we give a new perspective on the notion of a colored polysimplicial complex, originally described by Berkovich in  \cite[Section 3 and 4]{Berkovich_analytic}.

\begin{definition}
\label{polysimplex}
A \emph{(generalized) colored polysimplicial complex} is a functor $\Sigma : \Xi \rightarrow \mathbf{cPoly}$ from a small index category $\Xi$ to $\mathbf{cPoly}$ such that
\begin{enumerate}[(i)]
    \item for all $F \in \Xi$ and all $i:\Delta \rightarrow \Sigma(F)$ there exists an $F' \in \Xi$ and a map $\iota: F' \rightarrow F$ such that $\Sigma(F') = \Delta$ and $\Sigma(\iota) = i$,
    \item for all $F,G,H \in \Xi$ and $\alpha\colon G\rightarrow F$, $\beta\colon H\rightarrow F$ and $c\colon \Sigma(G)\rightarrow \Sigma(H)$ with 
    
   \[
\begin{tikzcd}
{} & \Sigma(F)  \\
\Sigma(G) \arrow{ur}{\Sigma(\alpha)} \arrow{rr}{c} && \Sigma(H) \arrow[swap]{ul}{\Sigma(\beta)}
\end{tikzcd}
\]
commutative, there exists a unique $\gamma : G \rightarrow H$ such that $\Sigma(\gamma) = c$, and
    \item the functor $\Sigma\colon \Xi\rightarrow\mathbf{cPoly}$   is faithful.
\end{enumerate}
If, instead of (iii), we require the stronger condition that
\begin{enumerate}
    \item[(iii)'] all objects $F\in \Xi$ have no automorphisms but the trivial one,
\end{enumerate}
we say that $\Sigma$ is \emph{without self-gluing}.
\end{definition}
Formally, this makes a colored polysimplicial complex a category fibered in groupoids over $\mathbf{cPoly}$. 
We think of this definition in the following way. The category $\Xi$ is an index category specifying how faces should be glued together. The first condition says that whenever $\Sigma(F)$ is part of our colored polysimplicial complex, so are its faces. The second condition says that, given two faces which are a face of a colored polysimplex, with one contained in the other, then the index category $\Xi$ has a unique arrow corresponding to the inclusion. Condition (iii) ensures that the only automorphisms in $\Xi$ are the ones that are in one-to-one correspondence with certain self-gluings of colored polysimplices in $\Sigma$.
Condition (iii)' prohibits the relative interior of a colored polysimplex in $\Sigma$ from being glued to itself and makes $\Sigma$ into a category fibered in sets. 

\begin{example}
\label{ex:torus}
We will construct the torus as a colored polysimplicial complex. Let $\Xi$ be the category with objects $\{a,A,B,S\}$, pictured in Figure \ref{fig:torus_index}.
Now, we must make a functor $\Sigma : \Xi \rightarrow \mathbf{cPoly}$. The object $a$ is sent to the standard zero dimensional polysimplex (vertex), the objects $A$ and $B$ are sent to the standard $((n_1),(a_1)) = ((1),(1))$ polysimplex (interval of length 1), and $S$  is sent to the standard $((1,1),(1,1))$ polysimplex (square with side lengths 1). The arrows are as described in Figure \ref{fig:torus_polysimplicial}.

\begin{figure}[h]
\centering
\begin{minipage}{.5\textwidth}
  \centering
  \includegraphics[height = 2 in]{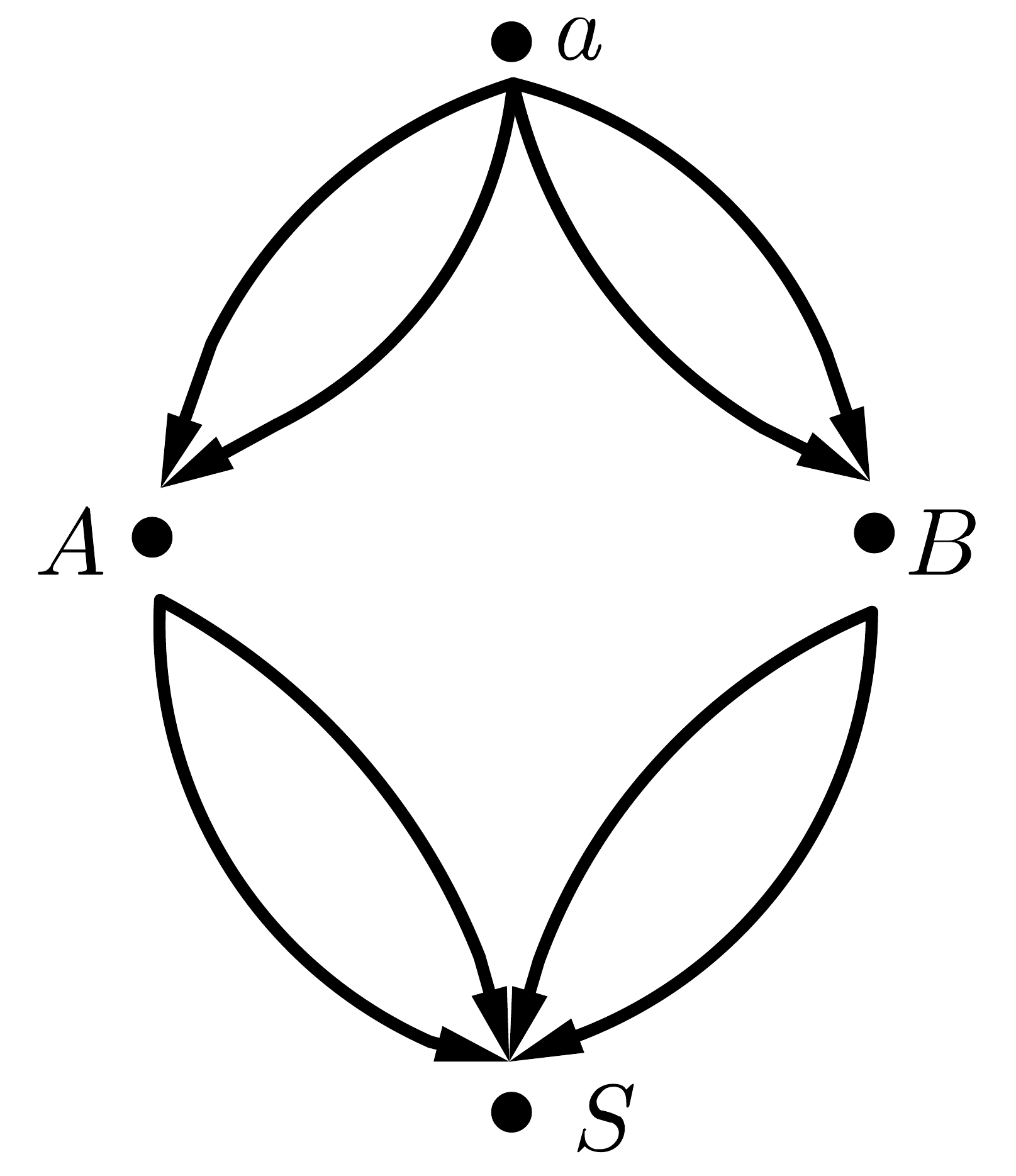}
  \captionof{figure}{The index category for Example \ref{ex:torus}.}
   \label{fig:torus_index}
\end{minipage}%
\begin{minipage}{.5\textwidth}
  \centering
  \includegraphics[height = 2 in]{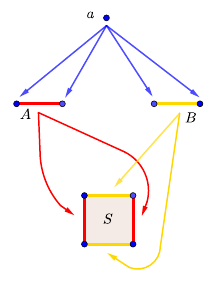}
  \captionof{figure}{The torus as a colored polysimplicial complex.}
  \label{fig:torus_polysimplicial}
\end{minipage}
\end{figure}
\end{example}

\begin{example}
\label{ex:mobius_index}
We will now see how to construct the M\"{o}bius band as a colored polysimplicial complex. Let $\Xi$ be the category with objects $\{a,A,B,T\}$ pictured in Figure \ref{fig:mobius_index}.
Now, we must make a functor $\Sigma : \Xi \rightarrow \mathbf{cPoly}$. The object $a$ is sent to the standard zero dimensional polysimplex (vertex), the objects $A$ and $B$ are sent to the standard $((1),(1))$ polysimplex, and $T$  is sent to the standard $((n_1),(a_1)) = ((2),(1))$ polysimplex (triangle with volume 1). The arrows are as described in Figure \ref{fig:mobius_polysimplicial}.

\begin{figure}
\centering
\begin{subfigure}{.3\textwidth}
  \centering
\includegraphics[height = 2 in]{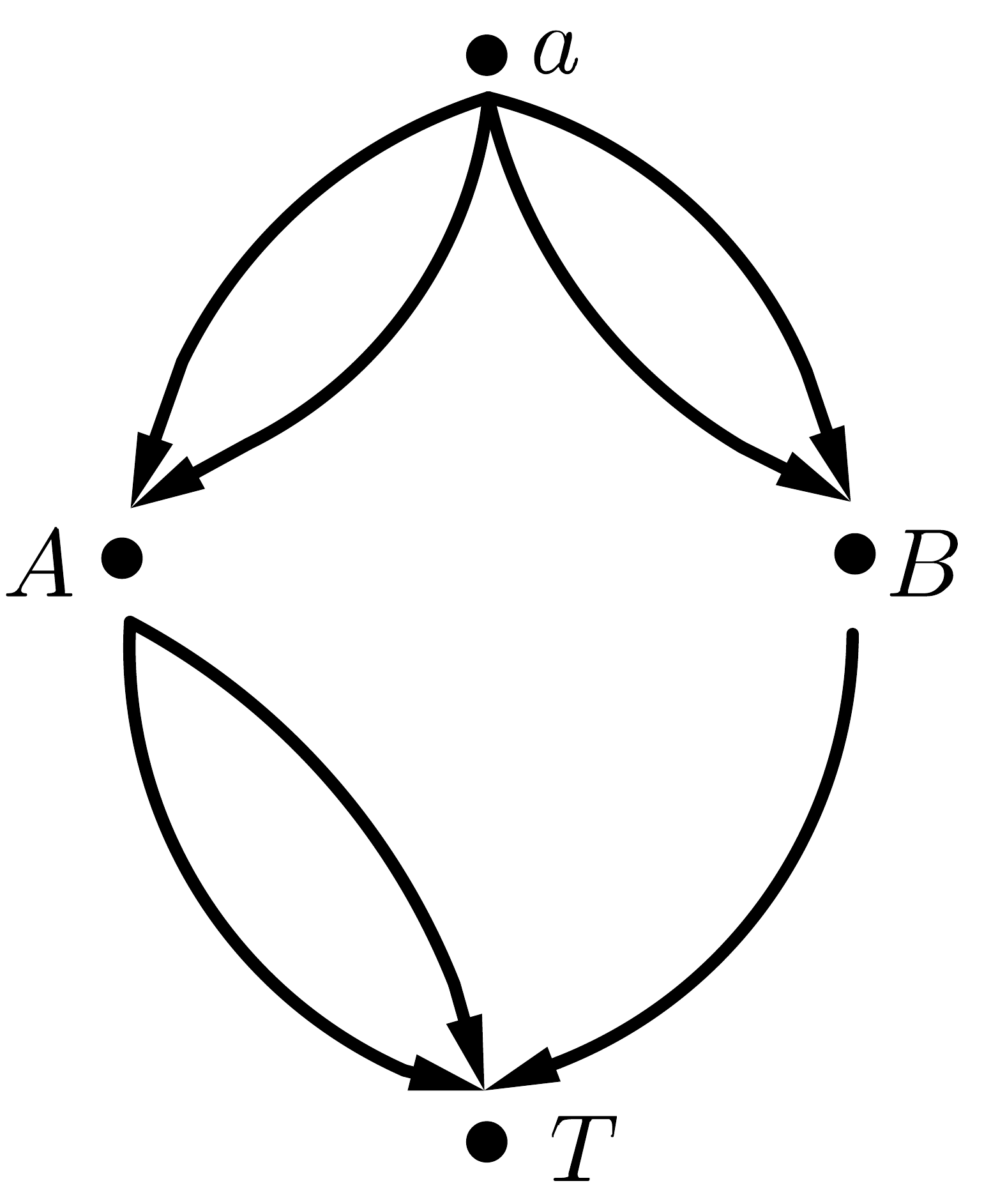}
  \caption{The index category for Example \ref{ex:mobius_index}.}
   \label{fig:mobius_index}
\end{subfigure}%
\hspace{0.2 in}
\begin{subfigure}{.3\textwidth}
  \centering
  \includegraphics[height = 2 in]{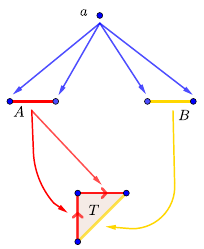}
  \caption{The M\"{o}bius band as a colored polysimplicial complex.}
  \label{fig:mobius_polysimplicial}
\end{subfigure}
\hspace{0.2 in}
\begin{subfigure}{.3\textwidth}
  \centering
\includegraphics[height=2in]{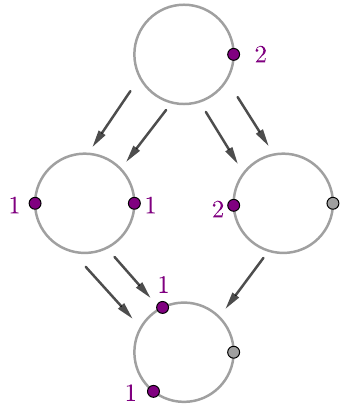}
  \captionof{figure}{The category of stable pairs $\Xi_2(G)^{op}$ in Example \ref{ex:stable_pairs_circle}.}
  \label{fig:stable_pairs_circle}
\end{subfigure}
\caption{The M\"{o}bius band is realized as a colored polysimplicial complex, as this is the symmetric power $(S^1)_2$.}
\label{fig:mobius}
\end{figure}

\end{example}


Given a polysimplicial complex $\Sigma$ indexed by $\Xi$, its geometric realization $|\Sigma|$ is obtained by gluing the disjoint union of the standard polysimplices associated to each polysimplex in $\Sigma(\Xi)$ along the images of the face morphisms. In other words, $|\Sigma|$ is the colimit of the functor which takes each element in $\Xi$ to the associated standard polysimplex, and takes each face morphism $\phi_{\alpha \beta} : F_\alpha \rightarrow F_\beta$ to the unique affine linear embedding taking $\Sigma(F_\alpha)$ to the corresponding face of $\Sigma(F_\beta)$.

\begin{remarks}\hfill
\begin{enumerate}[(i)]
\item Notice that Condition (iii) or (iii)' in Definition \ref{polysimplex} would also both imply the uniqueness in Condition (ii). We prefer to keep our axioms this way to ease the conversion to the language of categories fibered in groupoids. 
\item If in the index category $\Xi$ there is at most one arrow between any two objects, i.e. if it is \emph{thin}, the index category is naturally equivalent to (the category associated to) a poset. In this case no polysimplex in $\Sigma$ is glued to itself. See Example \ref{ex:loopless} below.
\item If we had dropped condition (iii) in Definition \ref{polysimplex} altogether we would get a notion of a \emph{colored polysimplicial stack}, in analogy with \cite{CavalieriChanUlirschWise_tropstack}.
\item In \cite{ChanGalatiusPayne} the authors introduce the notion of a \emph{symmetric $\Delta$-complex}. If in Definition \ref{polysimplex} we had considered a functor $\Sigma\colon \Xi\rightarrow I$ to the category $I$ of finite sets $[p]=\{0,\ldots,p\}$ with inclusions as morphisms (i.e. the category of abstract simplices) with the same conditions (i)-(iii), our definition would be equivalent to that of a symmetric $\Delta$-complex in the sense of \cite{ChanGalatiusPayne}. 

\end{enumerate}
\end{remarks}


\subsection{Divisors on Tropical Curves} 

Let $G = (V,E)$ be a graph. A \emph{length function} is a function $\vert.\vert : E \rightarrow \mathbb{R}_{> 0}$. We say that two tuples $(G,\vert.\vert)$ and $(G',\vert.\vert')$ are equivalent if there exists a common length preserving refinement. An equivalence class of such tuples is called a \emph{metric graph}. We usually identify a metric graph with its \emph{geometric realization}, the  metric space that is given by gluing intervals of length $\vert e\vert$ according to the incidences in $G$. A \emph{tropical curve} is a metric graph $\Gamma$ together with a weight function $h: \Gamma \rightarrow \mathbb{Z}_{\geq 0}$ which has finite support. A \emph{model} for $\Gamma$ is a pair $(G,\vert.\vert)$ which represents $\Gamma$ as a metric graph, and the support of $h$ is contained in the vertex  set of $G$. We define the \emph{genus} of a vertex-weighted graph $(G,h)$ to be $g(G,h):=b_1(G)+\sum_{v\in V(G)}h(v)$. The \emph{genus} of a tropical curve $\Gamma$ is the genus of one of its models. 

A \emph{divisor} on a tropical curve is a finite formal sum $D = \sum a_i p_i$ of points $p_i \in \Gamma$ with integer coefficients $a_i$. The \emph{degree} of a divisor $D = \sum a_i p_i$ is defined to be $\text{deg}(D) = \sum a_i$. A divisor $D$ is \emph{effective} if $D(p):=\sum_{p_i = p} a_i \geq 0$ for all $p \in \Gamma$. Given a tropical curve $\Gamma$, we denote by  $\text{Div}_d^+(\Gamma)$ the set of effective divisors of degree $d$ on $\Gamma$.

Define the $d$-th symmetric product $\Gamma_d$ of $\Gamma$ to be the quotient of the $d$-fold product $\Gamma^d$ of $\Gamma$ by the action of $S_d$ which permutes the factors.

\begin{lemma} We have a set-theoretic equality $\text{Div}_d^+(\Gamma) = \Gamma_d$.
\end{lemma}
\begin{proof}
Let $p \in \Gamma_d$. Then there is a representative $(p_1, \ldots, p_d) \in \Gamma^d$ for $p$. Consider the map $\phi: \Gamma_d \rightarrow \text{Div}_d^+(\Gamma)$, where
$$
(p_1, \ldots, p_d) \mapsto \sum_{i = 1}^d p_i.
$$
Then $\phi$ is well defined, because for any permutation $\sigma \in S_d$, we have $\sum_{i = 1}^d p_i = \sum_{i = 1}^d p_{\sigma(i)}$. The map $\phi$ is surjective, because given any effective divisor $D$ of degree $d$, we may write it in the form $\sum_{i = 1}^d p_i$. Then, we see that $\phi$ is injective because if $p_1 + \cdots + p_d = q_1 + \cdots + q_d$, then there is a permutation $\sigma \in S_d$ such that $p_i = q_{\sigma(i)}$.
\end{proof}

\subsection{$\text{Div}_d^+(\Gamma)$ as a colored polysimplicial complex} Let $(G,h)$ be a weighted graph. A \emph{weighted edge contraction} $\phi'\colon (G',h')\rightarrow (G,h)$ is an edge contraction $G'\rightarrow G$ such that $h'(\phi^{-1}(v))=h(v)$ for all $v\in V(G)$. We say that a weighted edge contraction $\phi\colon(G',h')\rightarrow (G,h)$ is a \emph{chain contraction} if it is given by contracting chains of edges that contain only 
vertices $v$ with $h'(v)=0$. We refer to the inner vertices of these chains as the \emph{exceptional vertices}. A chain contraction $\phi$ defines a map $\phi_* : \text{Div}_d^+(G') \rightarrow \text{Div}_d^+(G)$ by $\sum_{v \in V(G')} n_v v \mapsto \sum_{v \in V(G)} ( \sum_{v' \in \phi^{-1}(v)} n_{v'} )v$.

Let $\Gamma$ be a tropical curve and fix a model $(G,\vert.\vert)$ of $\Gamma$. We show that, associated to this data, there is a natural polysimplical complex $\Delta(G,d)$ 
whose geometric realization is equal to $\Div_d^+(\Gamma)=\Gamma_d$. 

\begin{definition}
A \emph{stable pair} of degree $d$ over $G$ is a tuple $(\phi\colon G'\rightarrow G,D)$ consisting of a  chain contraction $\phi\colon G'\rightarrow G$ together with an effective divisor $D\in\Div_{d}(G')$ such that $D(v)>0$ for all exceptional vertices of $G'$. 
\end{definition}

In our notation we typically suppress the reference to $\phi$ and only write $(G',D)$ instead of $(\phi\colon G'\rightarrow G,D)$. Denote by $\Xi_d(G)$ the category whose objects are stable pairs over $G$ of degree $d$ and whose morphisms  $(G_1',D_1)\rightarrow (G_2',D_2)$ are chain contractions $\phi_{12}\colon G_1'\rightarrow G_2'$ such that $\phi_2\circ\phi_{12}=\phi_1$ and $\phi_{12,\ast} D_1=D_2$. 

\begin{example}
\label{ex:stable_pairs_circle}
Let $\Gamma$ be the circle with circumference 1, and let $G$ be the graph with one edge and one vertex giving a model for $\Gamma$. Then $\Xi_2(G)^{op}$ is as pictured in Figure \ref{fig:stable_pairs_circle}. 

\end{example}

Consider now the functor
\begin{equation*}
    \Sigma_{(G,d)}\colon \Xi_d(G)^{op}\longrightarrow \mathbf{cPoly}
\end{equation*}
that associates to $(G',D)$ a colored polysimplex 
\begin{equation}
\label{eqn:deltagd}
    \Sigma_{(G,d)}(G',D):=\Delta(G',D):=\prod_{e\in E(G)}\Delta\big(n_e,\vert e\vert \big) \ ,
\end{equation}
where $n_e$ denotes the number of exceptional vertices in $G'$ that are contracted to the endpoints of $e$. Notice that $\Delta\big(n_e,\vert e\vert\big)$ is a point whenever $n_e=0$ and so these factors do not contribute to the product in Equation \eqref{eqn:deltagd}.
This way we may think of a point in $\Delta(G',D)$ as a tuple $(l_{e'})_{e'\in E(G')}$ of nonnegative numbers subject to the conditions 
\begin{equation*}
    \sum_{e'}l_{e'}=\vert e\vert
\end{equation*}
for all edges $e$ of $G$, where the sum is taken over all edges $e'$ for which $\phi(e')=e$ or an endpoint of $e$. 
As we will see in Proposition \ref{prop_sympow=polysimcplx}, the polysimplex $\Delta(G',D)$ parameterizes all divisors on $\Gamma$ whose combinatorics is described by the stable pair $(G',D)$.
 
 The functor $\Sigma_{(G,d)}$ sends an arrow $(G_1',D_1)\rightarrow (G_2',D_2)$ corresponding to the chain contraction $\phi_{12}$ in $\Xi_d(G)$ to the face morphism ${\Delta(G'_2,D_2)\hookrightarrow \Delta(G_1',D_1)}$ whose image is given by setting those $x_{e'}=0$ that correspond to edges in $G_1'$  that are contracted by $\phi_{12}$.

\begin{proposition}\label{prop_sympow=polysimcplx}
The functor $\Sigma_{(G,d)}$ defines a colored polysimplicial complex without self-gluing whose geometric realization is in natural bijection with $\Div_d^+(\Gamma)$.
\end{proposition}

From now on we always implicitly fix a semistable model $(G,\vert.\vert)$ of $\Gamma$ and, in a slight abuse of notation, denote $\Sigma_{(G,d)}$ by $\Gamma_d$. 

\begin{proof}[Proof of Proposition~\ref{prop_sympow=polysimcplx}]
The arrow $(G_1',D_1)\rightarrow (G_2',D_2)$ in $\Xi_d(G)$ naturally induces a face morphism ${\Delta(G'_2,D_2)\hookrightarrow \Delta(G_1',D_1)}$ onto the face of $\Delta(G_1',D_1)$ that is given by setting those $x_{e'}=0$ that correspond to edges in $G_1'$ that are contracted by $\phi_{12}$.

Conversely, given a face  $F$ of $\Delta(G_1',D_1)$, there is a collection of edges $e'$ of $G'_1$ for which $x_{e'}=0$ holds in the image of $F$. Let $\phi_{12}\colon G_1'\rightarrow G_2'$ be the chain contraction that contracts exactly those edges and set $D_2=\phi_{12,\ast} D_1$. Then $\Delta(G_2',D_2) = F$ and the arrow $(G_1',D_1)\rightarrow (G_2',D_2)$ is sent to the face morphism $F \hookrightarrow \Delta(G_1',D_1)$. This verifies Axiom (i) in Definition \ref{polysimplex} and, since $\phi_{12}$ is uniquely determined by the face $F$ also Axiom (iii), i.e. the fact that $\Sigma_{(G,d)}$ is faithful. 

Suppose now we have $(G_1',D_1)$, $(G_2',D_2)$, and $(G_3',D_3) \in \Xi_d(G)$ with $\alpha: (G_1',D_1)\rightarrow (G_2',D_2)$, $\beta : (G_1',D_1)\rightarrow (G_3',D_3)$, and $c \colon \Delta(G_3',D_3) \rightarrow \Delta(G_2',D_2)$ so that $\Sigma_{(G,d)}(\alpha)\circ c=\Sigma_{(G,d)}(\beta)$. This implies that there are edges in $G_1'$ that are not contracted by $\phi_{12}$ but are contracted by $\phi_{13}$. We now define $\gamma$ to be the contraction of these edges. Then $\Sigma_{(G,d)}(\gamma) = c$ and $\gamma$ is uniquely determined by $c$; thus also Axiom (ii) holds.  

Finally, no objects in $\Xi_d(G)$ have automorphisms except the trivial one, since there are no automorphisms that commute with the chain contraction.
So, the functor $\Delta_{(G,d)}$ defines a colored polysimplicial complex without self-gluing.

Let $D$ be an effective divisor of degree $d$ on $\Gamma$. Since $\Gamma$ is semistable, there is a unique model $G'$ of $\Gamma$ that admits a (possibly not unique) chain contraction to $G$ such that for all exceptional vertices $v$ of $G'$ we have $D(v)>0$. Then $(G',D)$ naturally defines a point in the relative interior of $\Delta(G',D)$. Conversely,
for a point in the geometric realization of the polysimplicial complex $\Sigma_{(G,d)}$, there is a unique stable pair $(G',D)$ such that this point lies in the relative interior of $\Delta(G',D)$. So  we may write this point as $(x_{e'})_{e'\in E(G')}$ in the relative interior of the geometric realization of $\Delta(G',D)$ (with $x_{e'}>0$). The geometric realization of $(G')$ with edge lengths $x_{e'}>0$ is equal to $\Gamma$ and $D$ naturally defines an effective divisor of degree $d$ on $\Gamma$. Thus the geometric realization of $\Delta_{(G,d)}$ is in natural bijection with $\Div_d^+(\Gamma)$. 
\end{proof}

\begin{remark}
 Note that in this construction you can recover the degree $d$ only as the dimension of a maximal polysimplex $\Delta(G',D)$ in $\Sigma_{(G,d)}$, where $D(v)=1$ for all exceptional vertices $v$ of $G'$.
 \end{remark}
 
 \begin{example}
 \label{ex:edge}
 Let $\Gamma$ be a tropical curve consisting of one edge $e$ of length 
 $l$ connecting two vertices $v$ and $v'$. Write $G$ for the underlying graph of $\Gamma$. Consider an effective divisor $D$ on $\Gamma$ of degree 2. Then we are in one of the following cases:
 \begin{itemize}
     \item $D=2v$, i.e. the divisor $D$ is supported only at $v$.
     \item $D=2v'$, i.e. the divisor is supported only in $v'$.
     \item $D=v+v'$, i.e. the divisor is supported in both $v$ and $v'$.
     \item $D=v+p$ for a point $p$ in the relative interior of $e$.
     \item $D=v'+p$ for a point $p$ in the relative interior of $e$.
     \item $D=2p$ for a point $p$ in the relative interior of $e$.
     \item $D=p+q$ for two different points $p$ and $q$ in the relative interior of $e$.
 \end{itemize}
The first three cases correspond to three zero-dimensional simplices in $\Sigma_{(G,2)}$, the next three cases correspond to three one-dimensional simplices of length (i.e. color) $l$ in $\Sigma_{(G,2)}$, and the last case corresponds to a $2$-dimensional simplex $\Sigma_{(G,2)}$ of color $l$. See Figure \ref{fig:edge}.
\begin{figure}
    \centering
    \includegraphics[height=1.7 in]{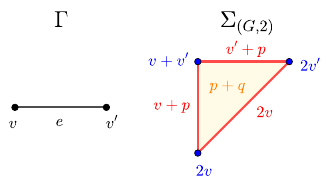}
    \caption{The graph $\Gamma$ and the associated colored polysimplex when $d=2$.}
    \label{fig:edge}
\end{figure}
 \end{example}

\begin{example} 
\label{interval_ex}
Consider the graph $G$ in Figure~\ref{interval_ex_G}.
\begin{figure}
  \centering
\includegraphics[width = 1.5 in]{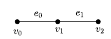}
  \captionof{figure}{The graph $G$ for Example~\ref{interval_ex}}
  \label{interval_ex_G}
\end{figure}
The spaces $G^2$ (the Cartesian product of $G$ with itself) and $G_2$ (the second symmetric product of $G$) are displayed in Figures~\ref{interval_ex_g2} and \ref{interval_ex_sym} respectively. In Figure~\ref{interval_ex_sym}, $G_2$ is displayed with the polysimplicial complex structure described in the proof of Proposition ~\ref{prop_sympow=polysimcplx}.

We note that this is not the quotient of  $\Gamma\times \Gamma$ by the natural $\Z_2$ operation.

\begin{figure}
\centering
\begin{minipage}{.45\textwidth}
  \centering
\includegraphics[height = 1.5 in]{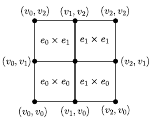}
  \captionof{figure}{The space $G^2$ in Example~\ref{interval_ex}}
  \label{interval_ex_g2}
\end{minipage}
\begin{minipage}{.5\textwidth}
  \centering
\includegraphics[height= 1.5 in]{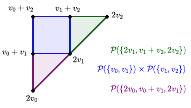}
  \captionof{figure}{The space $G_2$ in Example~\ref{interval_ex}}
  \label{interval_ex_sym}
\end{minipage}
\end{figure}
\end{example}

\begin{example} 
\label{ex:mobius_ex}
Consider the metric graph $S^1$, the unit circle, and let $G$ be the graph with one edge and one vertex giving a model for $S^1$. Then the second symmetric power $(S^1)_2$ is the M\"{o}bius band, see Figure \ref{fig:mobiusstrip}.

\begin{figure}
    \centering
    \includegraphics[height = 2.5in]{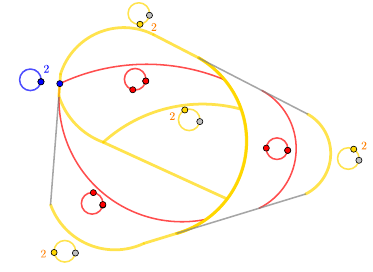}
    \caption{The second symmetric power $(S^1)_2$, as in Example \ref{ex:mobius_ex}.}
    \label{fig:mobiusstrip}
\end{figure}

\end{example}


\begin{example}
\label{ex:loopless}
Let $G$ be the dumbell graph, or the chain of two loops. In Figure~\ref{fig:dumbell} we give the category $\Xi_d(G)$. Since we have chosen a loopless model $G$, the category $\Xi_d(G)$ is thin. The polysimplicial complex $\Delta(G,2)$ has 15 maximal cells, five of which are triangles and 10 of which are squares. It has 25 edges and 10 vertices.
\end{example}

\begin{figure}
    \centering
    \includegraphics[width = \linewidth]{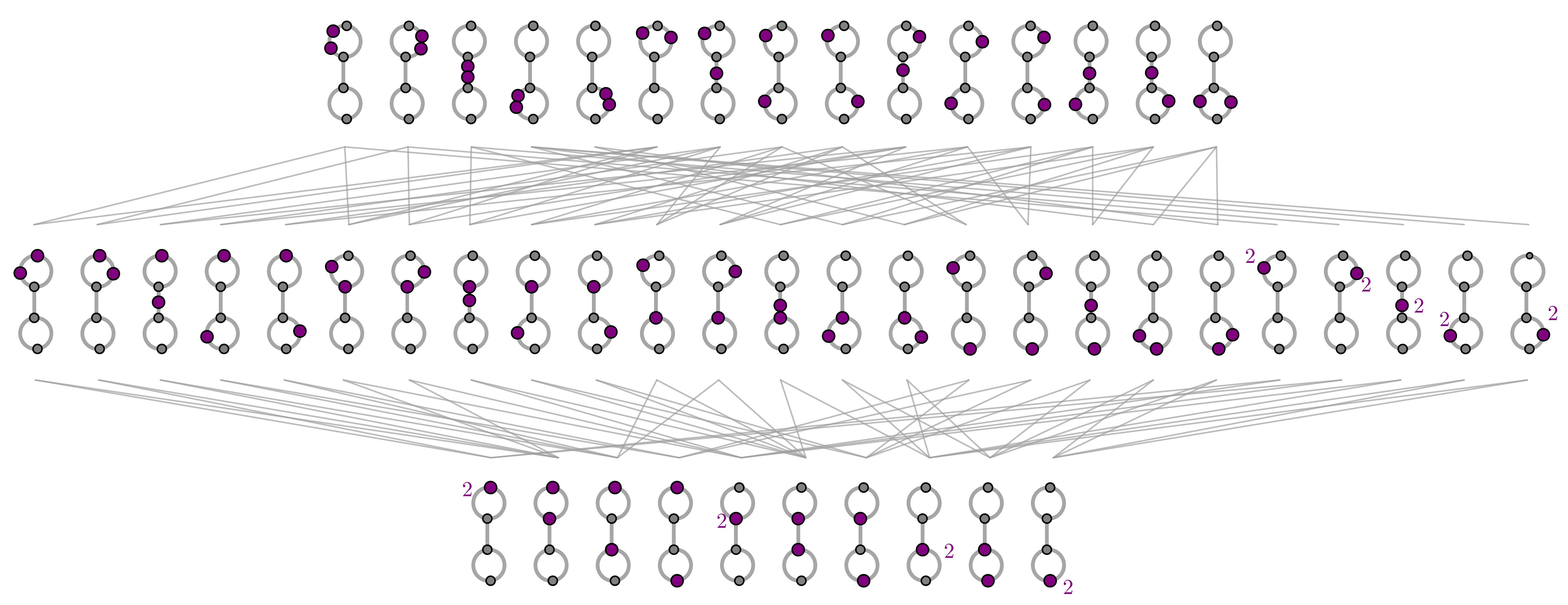}
    \caption{The poset $\Xi_d(G)$ for the dumbell graph $G$ and $d=2$.}
    \label{fig:dumbell}
\end{figure}



\section{Skeletons of polystable models -- a user's guide}\label{section_polystableskeletons}
Let $K$ be a non-Archimedean field, i.e. a field that is complete with respect to a fixed non-trivial non-Archimedean absolute value $\vert.\vert$. Denote by $R$ its valuation ring and by $k$ its residue field. 

\begin{definition}\label{definition_polystable}
Let $X$ be a smooth variety over $K$. A \emph{polystable model of $X$} is a flat and separated scheme $\calX$ over $R$, whose generic fiber is isomorphic to $X$ and such that for every geometric point $x$ of $\calX$ there is an \'etale neighborhood $\calU$ of $x$ in $\calX$ as well as an \'etale morphism $\gamma\colon \calU\rightarrow \Spec A_1\otimes_R\cdots \otimes_R A_r$ over $R$ where $A_i$ is of the form $R[t_1,\ldots, t_n]/(t_1\cdots t_k-a)$ for $a\in R$. If we may choose $\calU$ to be Zariski open for all $x$, we say that $\calX$ is a \emph{strict polystable model}. 
\end{definition}

Suppose now that $X$ is a scheme that is locally of finite type over $K$. In \cite{Berkovich_book} Berkovich has introduced a non-Archimedean analytic space $X^{an}$ associated to $X$. When $X=\Spec A$ is affine, a point $x\in X^{an}$ is a multiplicative seminorm $\vert .\vert_x\colon A\rightarrow \mathbb{R}_{\geq 0}$ that extends the non-Archimedean absolute value on $K$. This space carries the coarsest topology that makes the evaluation maps 
\begin{equation*}
\begin{split}
    \ev_f\colon X^{an}&\longrightarrow \R\\
    x&\longmapsto \vert f\vert_x
\end{split}
\end{equation*}
continuous for all $f\in A$. For a general $X$ we obtain $X^{an}$ by gluing the affine patches. We refer the interested reader to \cite{Berkovich_book} for full details on this construction. 

Suppose now that $X$ is proper over $K$ and that $\calX$ is a proper polystable model of $X$.  In \cite{Berkovich_analytic} Berkovich constructed a strong deformation retraction $\rho_\calX\colon\calX^{an}\rightarrow \Sigma(\calX)$ onto a closed subset of $\calX^{an}$ that naturally carries the structure of a colored polysimplicial complex, the \emph{non-Archimedean skeleton} associated to $\calX$. In this section we shall recall the basic properties of this construction. Our presentation is inspired by \cite[Section 4]{GublerRabinoffWerner} and \cite{Thuillier_toroidal}.

\begin{notation}
Let $\calX$ be a flat and separated model of $X$. Following \cite[Section 4.9]{Gubler_guide}, we denote by $\calX^\circ$ the analytic domain in $X^{an}$ consisting of those points that naturally extends to the model $\calX$. 
If $\calX=\Spec \calA$ is affine, this means we consider only those seminorms on $A=\calA\otimes K$ that are induced by  multiplicative  seminorms $\vert .\vert$ on $\calA$ that are \emph{bounded}, i.e. for which we have $\vert f\vert_x \leq 1$ for all $f\in\calA$. Note that, if $\calX$ is proper over $R$, then the valuative criterion for properness implies  $\calX^\circ=X^{an}$.
\end{notation}

\subsection{Tropicalization of a stable standard model}
Let $n\geq 0$, $k\leq n$, and $a\in R$. We refer to the affine $R$-scheme $\calZ(n,k,a)=\Spec A$ with 
\begin{equation*}
    A=R[t_0,\ldots, t_n]/(t_0\cdots t_k-a)
\end{equation*}
as a \emph{stable standard model}. Consider the standard simplex (as in Equation~\eqref{eq_coloredsimplex})
\begin{equation*}
    \Delta(k,a):=\Delta\big(k,\val(a)\big)=\big\{v\in\R_{\geq 0}^{k+1}\big\vert v_0+\ldots + v_k=\val(a)\big\} \ .
\end{equation*}
Here, we overload the notation by writing $\Delta(k,a)$ for $a \in R$ with $\Delta(k,\val(a))$.
There is a natural continuous tropicalization map
\begin{equation*}
    \trop_{n,k,a}\colon \calZ(k,n,a)^\circ\longrightarrow \Delta(k,a)
\end{equation*}
given by 
\begin{equation*}
x\longmapsto \big(-\log\vert t_0\vert_x, \ldots, -\log\vert t_k\vert_x\big) \ .
\end{equation*}
This map is well-defined, since 
\begin{equation*}
    -\log\vert t_0\vert_x - \ldots -\log\vert t_{k}\vert_x = -\log\vert t_0\cdots t_{k}\vert_x =-\log \vert a\vert_x
    = \val(a) \ .
\end{equation*}

\subsection{The skeleton of a stable standard model} The tropicalization map $\trop_{n,k,a}$ has a natural section $J_{n,k,a}\colon\Delta(k,a)\rightarrow \calZ(n,k,a)^\circ$ given by sending $v\in \Delta(k,a)$ to the multiplicative seminorm given by 
\begin{equation*}
    \sum_{\vec{l}\in\N^n} a_{\vec{l}} 
\big(t_0^{l_0}\cdots t_n^{l_n}\big)\longmapsto \max_{\vec{l}\in\N^n}\big(\vert a_{\vec{l}}\vert \cdot e^{-(l_0v_0+\cdots+l_kv_k)}\big) \ .
\end{equation*}
This is bounded because $a_{\vec{l}} \in R$ and $l_0v_0+\cdots+l_kv_k \geq 0$, by assumption. The section is well-defined, since 
\begin{equation*}
    J_{n,k,a}(v)(t_0\cdots t_n)=e^{-(v_0+\cdots +v_k)}=e^{-\val(a)}=\vert a\vert \ .
\end{equation*}
The composition $\rho_{n,k,a}=J_{n,k,a}\circ\trop_{n,k,a}$ defines a retraction map 
\begin{equation*}
    \rho_{n,k,a}\colon \calZ(n,k,a)^\circ\longrightarrow \calZ(n,k,a)^\circ
\end{equation*} onto a closed subset of $\calZ(n,k,a)^\circ$, the \emph{non-Archimedean skeleton} $\Sigma(n,k,a)$ of $\calZ(n,k,a)^\circ$.

\subsection{Tropicalization of a polystable standard model}\label{section_troppolystablestandardmodel}
Write $\vec{n}=(n_1,\ldots, n_r)$, $\vec{k}=(k_1,\ldots, k_r)$ as well as $\vec{a}=(a_1,\ldots, a_r)$ so that $k_i\leq n_i$. A \emph{polystable standard model} is an affine $R$-scheme of the form 
\begin{equation*}
    \calZ(\vec{n},\vec{k},\vec{a})=\calZ(n_1,k_1,a_1)\times_R\cdots\times_R\calZ(n_r,k_r,a_r) 
\end{equation*}
where each $\calZ(n_i,k_i,a_i)=\Spec A_i$ with $A_i=R[t^{(i)}_0,\ldots, t^{(i)}_{n_i}]/(t^{(i)}_0\cdots t^{(i)}_{k_i}-a_i)$ for $a_i\in R$.  The colored polysimplex associated to $\calZ(\vec{n},\vec{k},\vec{a})$ is defined to be (as in Equation~\eqref{eqn:deltanr})
\begin{equation*}
   \Delta(\vec{k}, \vec{a}):=\Delta\big(\vec{k},\val(\vec{a})\big)= \Big\{v\in \R_{\geq 0}^{k_1+\cdots+k_r+r}\Big\vert v_0^{(i)}+\ldots +v_{k_i}^{(i)}=\val(a_i) \textrm{ for all } i=1,\ldots r\Big\} \ .
\end{equation*}
There is a natural continuous tropicalization map 
\begin{equation*}
\trop_{\vec{n},\vec{k},\vec{a}}\colon \calZ(\vec{n},\vec{k},\vec{a})^{\circ}\longrightarrow \Delta(\vec{k}, \vec{a})
\end{equation*}
given by 
\begin{equation*}
    x\longmapsto \big(-\log\vert t^{(i)}_0\vert_x,\ldots,-\log\vert t^{(i)}_{k_i}\vert_x \big)_{i=1,\ldots, r} \ .
\end{equation*}
It is well-defined, since 
\begin{equation*}
    -\log\vert t^{(i)}_0\vert_x + \ldots +\log\vert t^{(i)}_{k_i}\vert_x = -\log\vert t^{(i)}_0\cdots t^{(i)}_{k_i}\vert_x 
    =-\log \vert a_i\vert_x 
    = \val(a_i) \ .
\end{equation*}

\subsection{The skeleton of a polystable standard model}
\label{section_skeletonpolystablestandardmodel} The tropicalization map has a natural section $J_{\vec{n},\vec{k},\vec{a}}\colon \Delta(\vec{k},\vec{a})\rightarrow \calZ(\vec{n},\vec{k},\vec{a})^\circ$. This is given by associating to $v\in\Delta(\vec{k},\vec{n})$ the bounded seminorm on $A_1\otimes_R\cdots\otimes_R A_r$ given by 
\begin{equation*}
    \sum_l  a_{l}\big(f_1^{(l)}\otimes \cdots\otimes f_r^{(l)}\big)\longmapsto \max_{l}\Big(\vert a_{l}\vert \cdot J_{n_1,k_1,a_1}(v^{(1)})(f_1^{(l)})\cdots J_{n_r,k_d,a_r}(v^{(r)})(f_r^{(l)}) \Big)
\end{equation*}
where $\sum_l  a_{l}\big(f_1^{(l)}\otimes \cdots\otimes f_r^{(l)}\big)$ for $f_i\in A_i$ denotes a general element in $A_1\otimes_R\cdots\otimes_R A_r$. 
The composition $\rho_{\vec{n},\vec{k},\vec{a}}:=J_{\vec{n},\vec{k},\vec{a}}\circ\trop_{\vec{n},\vec{k},\vec{a}}$ defines a retraction map 
\begin{equation*}
    \rho_{\vec{n},\vec{k},\vec{a}}\colon \calZ(\vec{n},\vec{k},\vec{a})^\circ\longrightarrow \calZ(\vec{n},\vec{k},\vec{a})^\circ
\end{equation*}
whose image is a closed subset in $\calZ(\vec{n},\vec{k},\vec{a})^\circ$, the \emph{non-Archimedean skeleton} $\Sigma(\vec{n},\vec{k},\vec{a})$ of $\calZ(\vec{n},\vec{k},\vec{a})^\circ$. 

\subsection{Stratification of a polystable model}

Given a polystable model $\calX$ of $X$, the special fiber $\calX_0$ admits a natural stratification by locally closed subsets, defined inductively as follows: we first write $\calX_0$ as a disjoint union
\begin{equation*}
\calX_0=\bigsqcup_{i=0}^n \calX_0^{(i)} \ .
\end{equation*}
Let $\calX_0^{(0)}$ be the open locus of regular points of $\calX_0$ and let $\calX_0^{(1)}$ be the open locus of regular points in $\calX_0-\calX_0^{(0)}$. In general, given $\calX_0^{(i)}$ for $i=1,\ldots, n$, we define $\calX_0^{(i+1)}$ to be the open locus of regular points in
\begin{equation*}
    \calX_0-\big(\calX_0^{(0)}\cup\cdots\cup\calX_0^{(i)}\big) \ .
\end{equation*}
The subsets $\calX_0^{(i)}$ are locally closed and smooth. We refer to the connected components of $\calX_{0}^{(i)}$ as \emph{the strata} of $\calX_0$.

\subsection{The skeleton of small open neighborhood}\label{section_skeletonsmallneighborhood}
Let $\calX$ be a polystable model of $X$. An \'etale open neighborhood $\delta\colon\calU\rightarrow \calX$ of a geometric point $x$ of a stratum $E$ is itself a polystable model of its generic fiber. The \'etale open neighborhood is said to be \emph{small} if $\calU$ is a strict polystable model of its generic fiber and the closure of all strata in the special fiber of $\calU$ contains $\delta^{-1}E$. 
We refer to a chart $\gamma\colon \calU\rightarrow \calZ(\vec{n},\vec{k},\vec{a})$ as in Definition~\ref{definition_polystable} as \emph{small} with respect to a stratum $E$ if $\delta\colon\calU\rightarrow \calX$ is a small \'etale open neighborhood of a geometric point in $E$ and the image of $\delta^{-1}E$ is contained in the closed stratum of $\calZ(\vec{n},\vec{k},\vec{a})$. 

Let $\calU$ be a small \'etale open neighborhood in $\calX$. In \cite{Berkovich_analytic} it is shown that there is a retraction $\rho_\calU\colon \calU^\circ \rightarrow \calU^\circ$ onto a closed subset $\Sigma(\calU)$ of $\calU^\circ$ such that, whenever $\gamma\colon \calU\rightarrow \calZ(\vec{n},\vec{k},\vec{a})$ is small chart, the diagram
\begin{center}
    \begin{tikzcd}
        \calU^\circ \arrow[rr,"\rho_\calU"]\arrow[d,"\gamma^\circ"] && \calU^\circ\arrow[d,"\gamma^\circ"]\\
        \calZ(\vec{n},\vec{k},\vec{a})^\circ \arrow[rr,"\rho_{\vec{n},\vec{k},\vec{a}}"]&& \calZ(\vec{n},\vec{k},\vec{a})^\circ
    \end{tikzcd}
\end{center}
commutes and the restriction of $\gamma^\circ$ to $\Sigma(\calU)$ induces a homeomorphism $\Sigma(\calU)\xrightarrow{\sim}\Sigma(\vec{n},\vec{k},\vec{a})$. The closed subset $\Sigma(U)$ is called the \emph{skeleton} of $\calU^\circ$. It carries the structure of a colored polysimplex: the injective homomorphism $\phi_\calU\colon \N^{\vert\vec{n}\vert}\hookrightarrow C^0(\Sigma(\calU),\R)$ is given by the restrictions of the pullbacks $\gamma^{-1}t^{(i)}_{j}$ to  $\Sigma(\calU)$ for $j=0,\ldots, k_i$. 

Given another small \'etale neighorhood $\calU'$ of $E$, we consider the fiber product $\calU\times_\calX\calU'$ and find a Zariski open subset $\calV$ that is a small \'etale open neighborhood of $E$. The charts $\gamma$ und $\gamma'$ of $\calU$ and $\calU'$ induce charts of $\calV$ and thus an isomorphism $\Sigma(\calU)\simeq \Sigma(\calU')$ (see \cite[Lemme 3.28 and Prop. 3.29]{Thuillier_toroidal} for an analogous argument in the trivially valued case). We therefore write $\Sigma_E=\Sigma(\calU)\simeq\Sigma(\calU')$ for the \emph{colored polysimplex associated to $E$}. 

The \'etale fundamental group $\pi_1(E,x)$ of the stratum $E$ acts by permuting the $\gamma^{-1}t_{0}^{(i)},\ldots, \gamma^{-1}t_{k_i}^{(i)}$ for each $i=1,\ldots, r$ and thus on $\Sigma(\calU)=\Sigma_E$ by automorphisms of colored polysimplices. As above, the image of this operation does not depend on the choice of $\calU$. We write $S_E$ for the image of this operation in $\Aut(\Sigma(\calU))=\Aut(\Sigma_E)$ and refer to this as the \emph{monodromy} group of $E$. 



\subsection{The polysimplicial complex of a polystable model}

Let $X$ be a smooth variety over $K$ and let $\calX$ be a polystable model of $X$ over the valuation ring $R$. We now associate to $\calX$ a polysimplicial complex $\Sigma(\calX)$. 

The index category $\Xi(\calX)$ is the \emph{category of strata} of $\calX_0$. Its objects are the strata $E$ of $\calX_0$, the endomorphisms of an object are given by the monodromy group $S_E$, and we have an arrow $E\rightarrow E'$ (with $E\neq E'$) for every \emph{\'etale specialization} from $E$ to $E'$. 
We refer the reader to \cite[Appendix A]{CavalieriChanUlirschWise_tropstack} for a precise definition of \'etale specializations. In our situation we can consider a component $\widetilde{E}$ of the preimage of $E$ in a small \'etale neighborhood $\calU$ of $E$. The \'etale specializations $E\rightarrow E'$ are then in one-to-one correspondence with the irreducible components of the preimage of $E'$ in $\calU$ that are in the closure of $\widetilde{E}$.  

We now define a functor $\Sigma\colon \Xi(\calX)\rightarrow \mathbf{cPoly}$ by $E\mapsto \Sigma_E$ on objects. An automorphism of $E$ in $\Xi(\calX)$ induces an automorphism of $\Sigma_E$ and an \'etale specialization $E\rightarrow E'$ induces a face morphism $\Sigma_E\rightarrow \Sigma_{E'}$. The avid reader may now check that the functor $\Sigma$ defines a colored polysimplicial complex by restriction to small \'etale neighborhoods.

\subsection{The skeleton of a polystable model}
Let $X$ be a smooth variety over $K$ and let $\calX$ be a polystable model of $X$ over the valuation ring $R$. In \cite{Berkovich_analytic} Berkovich has shown that the retraction maps $\rho_\calU$ on small \'etale open subsets naturally descend to a retraction map $\rho_\calX\colon \calX^\circ\rightarrow \calX^\circ$ such that the diagram
\begin{center}
    \begin{tikzcd}
        \calU^\circ \arrow[r,"\rho_\calU"]\arrow[d,"\subseteq"] & \calU^\circ\arrow[d,"\subseteq"']\\
        \calX^\circ \arrow[r,"\rho_\calX"] & \calX^\circ 
    \end{tikzcd}
\end{center}
commutes. The image of $\rho_\calX$ in $\calX^\circ$ is the \emph{non-Archimedean skeleton} of $\calX^\circ$. From this construction we immediately obtain that the skeleton is naturally homeomorphic to the geometric realization of the colored polysimplicial complex $\Sigma(\calX)$. In a slight abuse of notation we also use the notation $\Sigma(\calX)$ for the skeleton of $\calX$. 

In fact, the retraction map $\rho_\calX$ is actually a strong deformation retraction onto $\Sigma(\calX)$ using the natural torus operation on a polystable standard model and formally lifting them to $\calX$. Since this aspect of the construction will play no further role in the remainder of this article, we refer the avid reader to \cite{Berkovich_analytic} for   details.

\section{Skeletons of symmetric powers}

\subsection{A polystable model of $X_d$}

Let $X\rightarrow S$ be a scheme over $k$. Recall that a \emph{relative effective Cartier divisor $D$} on $X$ over $S$ is a closed subscheme of $X$ that is flat over $S$ and for which the ideal sheaf $I(D)$ is a line bundle. 
Let $d\geq 0$ and $2g-2+n>0$.

\begin{definition}
Let $\calDivbar_{g,\epsilon^n,d}$ be the category fibered in groupoids over schemes, whose objects are tuples $(\pi\colon X'\rightarrow S,\vec{p}', D)$ consisting of the following data:
\begin{enumerate}[(i)]
\item $\pi\colon X'\rightarrow S$ is a flat and proper morphism of connected nodal curves;
\item $\vec{p}'$ is an ordered collection of $\epsilon$-weighted sections $p_1',\ldots, p_n'\colon S\rightarrow X'$ that do not meet the nodes in each fiber $X'_s$ of $\pi$; and
\item $D$ is a relative effective Cartier divisor of degree $d$ on $X'$ over $S$, whose support does not intersect the nodes in each fiber $X'_s$ of $X'$ over $S$, such that the twisted canonical divisor 
\begin{equation*}
    K_\pi+\epsilon D +\epsilon p'_1 +\ldots + \epsilon p'_n
\end{equation*} is $\pi$-relatively ample for one and therefore all $0<\epsilon\leq\frac{1}{d+n}$.
\end{enumerate}
\end{definition}

We remind the reader of the useful characterization that a $\mathbb{Q}$-divisor $D$ on a nodal curve $X$ is \emph{ample}, if and only if the restriction of $D$ to every irreducible component of $X$ has positive degree. So Condition (iii) above ensures that on each rational component of $X$ there are at least three special points (i.e. nodes, marked points, or points in the support of $D$), two of which have to be nodes. This, in particular, means that the underlying curve is semistable.

Let $\epsilon>0$, such that  $\epsilon\leq\frac{1}{d+n}$ (and $<$, if $d+n=0$). Denote by $\calMbar_{g,\epsilon^{n+d}}$ the moduli space of weighted stable curves of genus $g$ with $n+d$ marked points of weight $\epsilon$ in the sense of \cite{Hassett}. This moduli space parametrizes pairs $(X,\vec{p})$ consisting of a projective nodal curve $X$ and $n+d$ marked, not necessarily distinct points $\vec{p}=(p_1,\ldots, p_{n+d})$ such that the divisor $K_X+\epsilon p_1+\ldots +\epsilon p_{n+d}$ is ample. Note that, this condition, in particular, forces the underlying curve to be semistable. 

There is a natural operation of $S_d$ on $\calMbar_{g,\epsilon^{n+d}}$ that permutes the second $d$ marked points of weight $\epsilon$. Then $\calDivbar_{g,\epsilon^n,d}$ is naturally equivalent to the relative coarse moduli space of 
\begin{equation*}
    \big[\calMbar_{g,\epsilon^n,\epsilon^d}/S_d\big]
\end{equation*}
over $\calMbar_{g,\epsilon^n}$ in the sense of \cite[Theorem 3.1]{AbramovichOlssonVistoli}. So, in particular, it is a smooth and proper Deligne-Mumford stack with a projective coarse moduli space. There is a natural forgetful-stabilization morphism $\calDivbar_{g,\epsilon^n,d}\rightarrow \calMbar_{g,\epsilon^n}$ and we write $\calDiv_{g,\epsilon^n,d}$ for its restriction to $\calM_{g,\epsilon^n}$. 

Let $\big[X,\vec{p}\big]$ be a point in $\calMbar_{g,\epsilon^{n+d}}$. Then, since $\calMbar_{g,\epsilon^{n+d}}$ has a normal crossing boundary by \cite[Theorem 1.1]{Ulirsch_tropHassett}, the formal completion of the local ring of $\calMbar_{g,\epsilon^{n+d}}$ at the point $\big[X,\vec{p}\big]$ may be written as the ring of power series $k\llbracket t_1,\ldots, t_{3g-3+n+d}\rrbracket$, where the last $n+d$ coordinates parametrize formal deformations of the marked points. Then, at the image $\big[X,p_1,\ldots, p_n, p_{n+1}+\ldots+p_{n+d}\big]$ of $\big[X,\vec{p}\big]$ in $\calDivbar_{g,\epsilon^n,d}$, the formal completion of the local ring is given by $k\llbracket t_1,\ldots, t_{3g-3+n+d}\rrbracket^{S_d}$, where $S_d$ operates on the last $d$ coordinates. Using elementary symmetric polynomials we see that this is isomorphic to $k\llbracket t_1,\ldots, t_{3g-3+n}, T_{1},\ldots, T_d\rrbracket$. Thus the complement of $\calDiv_{g,\epsilon^n,d}$ in $\calDivbar_{g,\epsilon^n,d}$ has (stack-theoretically) normal crossings. 

\begin{remark}
For $n=0$, the moduli space $\calDivbar_{g,d}$ was constructed in \cite[Section 2]{MoellerUlirschWerner_realizability}. It is also equal to a special case of the moduli space of stable quotients, as defined in \cite[Section~4]{MarianOpreaPandharipande}. 
\end{remark}

\begin{proposition}\label{prop_fiber}
Let $(X,\vec{p})=(X,p_1,\ldots, p_n)$ be a stable $\epsilon$-weighted marked curve of genus $g$ given by a morphism $\Spec(k) \rightarrow \calMbar_{g,\epsilon^n}$. The fiber $\calDivbar_{g,\epsilon^n,d} \times_{\calMbar_{g,\epsilon^n}} \Spec(k)$ over this point is represented  by a fine moduli scheme $\overline{\Div}^+_d\big(X,\vec{p}\big)$
which parametrizes tuples $(X',\vec{p}',D)$ consisting of the following data:
\begin{enumerate}[(i)]
    \item a nodal curve $X'$;
    \item a collection of $\epsilon$-weighted marked points $\vec{p}'=(p_1',\ldots,p_n')$ of $X'$ such that $p_1',\ldots,p_n'$ do not meet the nodes of $X'$ together with an isomorphism from the stabilization of $(X',\vec{p}')$ to $(X,\vec{p})$;
    \item a relative effective Cartier divisor $D$ of degree $d$ on $X'$ whose support does not intersect the nodes or marked points of $X'$ such that the twisted canonical divisor
$$
K + \epsilon D + \epsilon p_1' + \cdots + \epsilon p_n'
$$
is ample, where $\epsilon = \frac{1}{d+n} > 0$.
\end{enumerate}
\end{proposition}

\begin{proof}
This has a fine moduli scheme, since the forgetful and stabilization map $\calDivbar_{g,n,d}\rightarrow \calMbar_{g,n}$ is representable. The interpretation of the fiber product in Parts (i) - (iii) is an immediate consequence of the definition of the $2$-fiber product.
\end{proof}

 If $X$ is smooth and does not have marked points, the space  $\overline{\Div}_{d}^+(X)$ gives effective divisors on $X$ and is the $d$-th symmetric power $X_d$ (see \cite[Theorem 3.13]{Milne_Jacobian}).
The natural forgetful morphism $\calDivbar_{g,\epsilon^n,d}\rightarrow\calMbar_{g,\epsilon^n}$ associates to $(X'\rightarrow S,\vec{p},D)$ the stabilization of $(X'\rightarrow S,\vec{p})$.

\begin{definition}
Let $X$ be a smooth projective curve of genus $g$ over $K$ and suppose that $\calX$ is a semistable model of $X$ over $\Spec R$ that admits a section. Let $n$ be the number of rational irreducible components in the special fiber of $\calX$ having 2 or fewer nodes. Choose $n$ marked sections $s_1,\ldots, s_n$ of weight $\epsilon$ of $\calX$ that do not meet the singularities in the special fiber, such that $(\calX,s_1,\ldots, s_n)$ is stable of type $(g,\epsilon^n)$.  This datum is specified by a morphism 
$\Spec R\rightarrow \calMbar_{g,\epsilon^n}$.
Define $\overline{\Div}_d^+(\calX)$ to be
\begin{equation*}
    \overline{\Div}_d^+(\calX,\vec{p})=\Spec R\times_{\calMbar_{g,\epsilon^n}}\calDivbar_{g,\epsilon^n,d} \ .
\end{equation*}
\end{definition}

\begin{proposition}
The scheme $\overline{\Div}_d^+(\calX,\vec{p})$ is a proper polystable model of $X_d$.
\end{proposition}

\begin{proof}

Since $\overline{\Div}_d^+(\calX)$ is the fibered product
$\Spec R\times_{\calMbar_{g,\epsilon^n}}\calDivbar_{g,\epsilon^n,d} 
$
it is is flat and proper over $R$.
The generic fiber of $\overline{\Div}_d^+(\calX)$ consists of effective degree $d$ divisors on the curve $X$. Since the divisor can possibly be supported on the points $\vec{p}$, this is nothing but the $d$-th symmetric power $X_d$ of $X$. So the generic fiber of $\overline{\Div}_d^+(\calX)$ is isomorphic $X_d$.

Points in the special fiber of $\overline{\Div}_d^+(\calX)$ are given by pairs $((\calX',\vec{p}'),D)$ where the special fiber $\calX_0'$ is semistable, $(\calX',\vec{p}')$ is an $\epsilon$-weighted $n$-marked curve, together with an isomorphism between the stabilization of  $(\calX_0',\vec{p}')$ and $(\calX_0,\vec{p})$, and $D$ is an effective divisor on $\calX'$ of degree $d$ whose restriction to the special fiber is supported in the non-singular locus of $\calX_0'$ and which has positive degree on every exceptional component of $(\calX_0',\vec{p}')$ (i.e. exceptional components of $\calX_0'$ that that do not contain a marked point). 

Consider the nodes of $\calX$ given by \'etale local equations $x_iy_i=a_i$ for $a_i\in R$ (for $i=1,\ldots, r$). Write the nodes in $\calX'$ above $x_iy_i=a_i$ as $x_j^{(i)}y_j^{(i)}=t_j^{(i)}$ for $j=1,\ldots, k_i$ and coordinates $t_j^{(i)}$ on $\overline{\Div}_d^+(\calX)$. In this case we have $a_i=t_1^{(i)}\cdots t_{k_i}^{(i)}$ on $\overline{\Div}_d^+(\calX)$. These coordinates can be chosen on a \'etale open neighborhood $\calU$ and we may add further coordinates  $t_{k_i+1}^{(i)}\ldots, t_{n_i}^{(i)}$ (for $i=1,\ldots, r$) so that, possibly after shrinking $\calU$, they define a small \'etale chart $\gamma\colon \calU\rightarrow \calZ(\vec{n},\vec{k},\vec{a})$. 
\end{proof}

We explicitly point out that the scheme $\overline{\Div}_d^+(\calX)$ is not the quotient of ${\calX\times_R\cdots \times_R\calX}$ by the operation of $S_d$. While not being smooth over $R$, the $R$-scheme $\overline{\Div}_d^+(\calX)$ only admits at most toroidal singularities over $R$, since, heuristically, we are allowed to perform a weighted blow-up in the special fiber of $\calX$ whenever the support of $D$ is at risk of meeting the singularities in $\calX_0$. 

\subsection{Stratification by dual graphs}
Consider a point in the special fiber of $\overline{\Div}_d^+(\calX)$. It is given by a pair $(X_0',D)$ where $X_0'$ is a nodal curve together with a morphism $X_0'\rightarrow X_0$ given by contracting rational components and $D$ is an effective divisor on $X'_0$ whose support is contained in the non-singular locus of $X_0'$ and which has positive degree on every exceptional component of $X_0'$. 

We associate to $(X_0',D)$ a dual stable pair $\big(G',\mdeg(D)\big)$ over $G$ as follows: The graph $G'$ is the weighted dual graph of $X_0'$. Its vertices $v$ correspond to the components $X_v'$ of $X_0'$ and it contains an edge $e$ emanating from two vertices $v$ and $v'$ for every node connecting the two components $X_v'$ and $X_{v'}'$. It is endowed with a natural vertex weight $h\colon V(G')\rightarrow \Z_{\geq 0}$ given by $h(v)=g(X_v)$, the genus of the component $X_v$. Finally the degree of the restriction of $D$ to every component $X_v$ defines a divisor
\begin{equation*}
    \mdeg(D)=\sum_{v\in V(G')} \deg \big(D\vert_{X_v}\big)\cdot v
\end{equation*}
on $G'$ supported on the vertices of $G'$, the \emph{multidegree of $D$}. 

The graph $G'$ is naturally a subdivision of $G$, the dual graph of $\calX_0$. The condition that $K_{X_0'}+D$ has non-empty intersection with every exceptional component of $X_0'$ over $X_0$ is equivalent to the condition that $D(v)>0$ for every exceptional vertex $v$ of $G'$ over $G$, i.e. to the condition that $(G',\mdeg(D))$ is a stable pair over $G$.

\begin{proposition}\label{prop_skel=sympow}
The colored polysimplicial complex $\Gamma_d$ is naturally isomorphic to $\Sigma\big(\overline{\Div}_d^+(\calX)\big)$.
\end{proposition}

\begin{proof}
The strata of the special fiber of $\overline{\Div}_d^+(\calX)$ are precisely the locally closed subsets on which the dual graphs are constant. In fact, the smooth locus $\calX_{d,0}^{(0)}$ of $\calX_{d,0}$ is the exactly locus of stable pairs $(X_0',D)$ for which $X_0'$ is isomorphic to $X_0$, which translates into the dual graph $G'$ of $X_0'$ being isomorphic to the dual graph $G$ of $X_0$. The different strata in $\calX_{d,0}^{(0)}$ are distinguished by the multidegree of $D$. 

Similarly, for $i=0,\ldots, d-1$ the regular locus of $\calX_{d,0}^{(i+1)}$ of  \begin{equation*}
    \calX_{d,0}-\big(\calX_{d,0}^{(0)}\cup \cdots\cup\calX_{d,0}^{(i)}\big)
\end{equation*} 
corresponds exactly to the locus of stable pairs $(X_0',D)$ that contain $i+1$ exceptional components. This translates into the condition that the dual graph contains exactly $i+1$ exceptional vertices over $G'$. The different strata, again, are distinguished by the multidegree of $D$. 

Moreover, notice that for every stable pair $(G',D)$ the locus $E_{(G',D)}$ of points in $\calX_{d,0}$ whose dual pair is $(G',D)$ is non-empty. The \'etale specializations $E_{(G',D')}\rightarrow E_{(G'',D)}$ between strata are in a natural one-to-one correspondence with chain contractions $\pi\colon G''\rightarrow G'$ over $G$ for which $\pi_\ast D'=D$. The automorphism group of a stratum is always trivial by part (ii) of Proposition \ref{prop_fiber}. So there is an order-preserving equivalence between the category  $\Xi(\overline{\Div}_d^+(\calX))$ of strata of $\calX_{d,0}$ and the category $\Xi_d(G)$ of stable pairs $(G',D)$ over $G$.

Finally, consider the nodes of $\calX$ given by by local equations $x_iy_i=a_i$ for $a_i\in R$ (for $i=1,\ldots, r$). Write the nodes in $\calX'$ above $x_iy_i=a_i$ as $x_j^{(i)}y_j^{(i)}=t_j^{(i)}$ for $j=1,\ldots, k_i$ and coordinates $t_j^{(i)}$ on $\overline{\Div}_d^+(\calX)$ and recall that in this case we have $a_i=t_1^{(i)}\cdots t_{k_i}^{(i)}$ on $\overline{\Div}_d^+(\calX)$. From this description we see that the colored polysimplex  $\Sigma_{(E_{(G',D)})}$ of the stratum $E_{(G',D)}$ is equal to $\Delta(\vec{k},\vec{a})$.

On the other hand, the $k_i$ are precisely the number of exceptional vertices over an edge $e_i$ of $G$ and the edge length $\vert e_i\vert$ of $e_i$ is equal to $\val(a_i)$. So the colored polysimplex $\Delta(G,D')=\Delta\big(\vec{k},\val(\vec{a})\big)$ of a stable pair $(G',D)$ is equal to $\Delta(\vec{k}, \vec{a})=\Sigma_{(E_{(G',D)})}$. This identification naturally commutes with the face morphisms induced by \'etale specializations $E_{(G',D)}\rightarrow E_{(G'',D)}$ and $(G'',D')\rightarrow (G',D)$ respectively and so we have found a canonical isomorphism between $\Gamma_d$ and the skeleton $\Sigma(\overline{\Div}_d^+(\calX))$ of $\overline{\Div}_d^+(\calX)$. 
\end{proof}

\subsection{The process of tropicalization} Let $X$ be a smooth projective curve over $K$ and suppose there is a fixed semistable model $\calX$ of $X$ over $R$ that admits as section. By the semistable reduction theorem, we can always find such a $\calX$ if we are willing to replace $X$ by its base change to a finite extension of $K$. 

Denote by $\Gamma$ the dual tropical curve of $\calX$. We now define the tropicalization map 
\begin{equation*}
    \trop_{X_d}\colon X_d^{an}\longrightarrow \Gamma_d \ .
\end{equation*}
A point $x$ in $X_d^{an}$ can be represented by a morphism $\Spec L\rightarrow X_d$ for a non-Archimedean extension $L$ of $K$. This, in turn, corresponds to an effective Cartier divisor $D$ of degree $d$ on $X_L$, via the interpretation of $X_d$ as a moduli space of effective divisors (see \cite[Theorem 3.13]{Milne_Jacobian}). Since $\overline{\Div}_d^+(\calX)$ is proper over $\Spec R$, the valuative criterion provides us with a unique semistable model $\calX'$ together with an isomorphism between its stabilization and $\calX$ and a relative effective Cartier divisor $\calD$ in $\calX'$ such that
\begin{itemize}
\item the generic fiber of $\calD$ is equal to $D$, 
\item the support $\supp(\calD_0)$ in the special fiber does not meet the nodes of $\calX'_0$, and
\item $\supp(\calD_0)\cap E\neq\emptyset$ for every exceptional component $E$ of $\calX'_0$ over $\calX_0$. 
\end{itemize}
We may now define $\trop_{X_d}(x)$ to be the divisor that arises as the multidegree of $\calD_0$ on $\Gamma$. It is naturally supported on the model $G'$ of $\Gamma$ given by the dual graph of $\calX_0'$. 

A posteriori, Theorem~\ref{thm_skeleton=symmetricpower} implies that the construction of $\trop_{X_d}$ does not depend on any of the above choices and that the tropicalization map is invariant under base change. In other words, given a non-Archimedean extension $K'$ of $K$ the dual tropical curve $\Gamma_{X_{K'}}$ is naturally isometric to $\Gamma_X$ and the natural diagram 
\begin{equation*}
    \begin{tikzcd}
        X_{K'}^{an}\arrow[rr,"\trop_{X_{K',d}}"]\arrow[d]&& \big(\Gamma_{X_{K'}}\big)_d\arrow[d,"\simeq"]\\
        X^{an}\arrow[rr,"\trop_{X_d}"']&& \Gamma_d
    \end{tikzcd}
\end{equation*}
commutes. 

\subsection{The specialization map}\label{section_specialization}
Denote by $\Div_{K,d}(X)$ the group of \emph{$K$-split} divisors on $X$, which may be written $\sum_{i=1}^n a_i p_i$ for points $p_i\in X(K)$ and $\sum_i a_i = d$. Write $\rho_X\colon X^{an}\rightarrow \Gamma$ for the retraction of $X^{an}$ to $\Gamma_{X}\simeq\Gamma$, which can be thought of as the skeleton of $X^{an}$. In \cite{Baker_specialization}, Baker constructs a   \emph{specialization homomorphism}
\begin{equation*}
    \rho_{X,\ast}\colon \Div_{K,d}(X)\longrightarrow \Div_d(\Gamma) \ .
\end{equation*}
It is defined by sending a $K$-split divisor $D=\sum_{i=1}^n a_i p_i$ on $X$ to the divisor 
\begin{equation*}
    \rho_{X,\ast}(D)=\sum_{i=1}^n a_i\rho_X(p_i)
\end{equation*}
on $\Gamma$. 

Recall that, if $Y$ is a scheme locally of finite type over $K$, then there is a natural injective map  $i\colon Y(K)\hookrightarrow Y^{an}$, whose image is dense if $K$ is algebraically closed. On an affine patch $U=\Spec A$ it is given by associating to a $K$-rational point the multiplicative seminorm
\begin{equation*}
    A\longrightarrow K\xlongrightarrow{\vert.\vert_K} \R_{\geq 0} 
\end{equation*}
on $A$. So, in particular, there is a natural injective map 
\begin{equation*}
    i\colon \Div_{K,d}^+(X)\hooklongrightarrow X_d^{an}
\end{equation*}
whose image is dense in $X_{d}^{an}$ if $K$ is algebraically closed.

\begin{proposition}\label{prop_spec=trop}
Given a $K$-split effective divisor $D$ on $X$ of degree $d$, we have 
\begin{equation*}
\trop_{X_d}\big(i(D)\big)=\rho_{X,\ast}(D) \ .
\end{equation*}
 
\end{proposition}

In other words, the natural diagram 
\begin{center}
    \begin{tikzcd}
        \Div_{K,d}^+(X)\arrow[r,"\rho_{X,\ast}"] \arrow[d,hookrightarrow,"i"'] &\Div_d^+(\Gamma)\arrow[d,"\simeq"]\\
        X_{d}^{an} \arrow[r,"\trop_{X_d}"] & \Gamma_d
    \end{tikzcd}
\end{center}
commutes. 

\begin{proof}[Proof of Proposition~\ref{prop_spec=trop}] 
Suppose first that $K=\overline{K}$ is algebraically closed. Let $D=\sum_{i=1}^na_ip_i$ be an effective $K$-split divisor on $X$. Since $\overline{\Div}_d^+(\calX)$ is proper over $R$, we find a unique semistable model $\calX'$ over $\calX$ as well as a relative effective Cartier divisor $\calD$ on $\calX'$ such that
\begin{itemize}
\item the generic fiber of $\calD$ is equal to $D$, 
\item the support $\supp(\calD_0)$ in the special fiber does not meet the nodes of $\calX'_0$, and
\item $\supp(\calD_0)\cap E\neq\emptyset$ for every exceptional component $E$ of $\calX'_0$ over $\calX_0$. 
\end{itemize}

By \cite[Theorem 4.11]{BakerPayneRabinoff_structure} (also see \cite[Theorem 4.3.1]{Berkovich_book}), the semistable model $\calX'$ gives rise to a semistable vertex set $V$ in $X^{an}$, i.e. a set of points $v$ in $X^{an}$ whose complement is a collection of closed pointed discs and annuli. The vertices $v$ are precisely the vertices in the dual graph of $\calX_0$ and the edges of the dual graph correspond to the annuli in $X^{an}-V$.   

 Let $r\colon X^{an}\rightarrow X_k$ be the reduction map. Then the pointed discs in $X^{an}$ are given by $B(v)=r^{-1}(U_v)-V$ where $U_v$ is the open subset of a component in $\calX'_0$ given by removing all of its nodes, and the annuli are given by $B(e)=r^{-1}(x_e)$, where the $x_e$ are the nodes of $\calX'_0$. The restriction of the retraction map $\rho_X$ to a pointed disc $B(v)$ shrinks all points $B(v)$ to the corresponding point $v\in\Gamma$ and its restriction to an annulus $B(e)$ is given by the retraction of the annulus to its skeleton which is isometric to $e$.

So, if the point $p_i$ extends to a component $X_{0,v_i}'$ of $\calX_0'$ via $\calD$, its reduction is an element of $U_{v_i}$. Therefore $p_i$ is a point of $B(v_i)$ and thus $\rho_X(p_i)=v_i$. So, by linearity, we have:
\begin{equation*}
    \rho_{X,\ast}(D)=\sum_{i=1}^na_i v_i=\sum_{v\in V(G')} \deg\big(D\vert_{X_{v}}\big) \cdot v =\mdeg_{\calX'_0} (\calD_0) \ .
\end{equation*}

The general case, when $K$ may not be algebraically closed, follows from the invariance of $\trop_{X_d}$ and $\rho_X$ under base change by a non-Archimedean extension $K'$ of $K$.
\end{proof}

\subsection{Proof of Theorem~\ref{thm_skeleton=symmetricpower}}
We conclude this section with the proof of Theorem~\ref{thm_skeleton=symmetricpower}.

\begin{proof}[Proof of Theorem~\ref{thm_skeleton=symmetricpower}]
By Propositions~\ref{prop_sympow=polysimcplx} and~\ref{prop_skel=sympow}, there is a natural isomorphism $\mu_{X_d}\colon\Gamma_d\xrightarrow{\sim}\Sigma(X_d)$. What remains to show is that the diagram
\begin{center}\begin{tikzcd}
 X_d^{an} \arrow[dr,"\rho_{X_d}"'] \arrow[drrr, bend left, "\trop_{X_d}"]& & &\\
 & \Sigma(X_d) & & \Gamma_d \arrow[ll, "\mu_{X_d}","\sim"']
\end{tikzcd}\end{center}
commutes. 

Consider the nodes of the special fiber of $\calX$ given by by local equations $x_iy_i=a_i$ for $a_i\in R$ (for $i=1,\ldots, r$). Write the nodes in $\calX'$ above $x_iy_i=a_i$ as $x_j^{(i)}y_j^{(i)}=t_j^{(i)}$ for $j=1,\ldots, k_i$ and coordinates $t_j^{(i)}$ on $\overline{\Div}_d^+(\calX)$. Write $\calD$ for the relative effective Cartier divisor on $\calX'$ that extends $D$ on $X$. Then the tropicalization $\trop_{X_d}(D)$ is given by the following data: 
\begin{itemize}
    \item the dual graph $G_{\calX'}$ of the special fiber of $\calX'$;
    \item the chain contraction $G_{\calX'}\rightarrow G_{\calX}$ that is given by stabilization and the given isomorphism $(\calX')_{st}\xrightarrow{\sim}\calX$;
    \item the edge length of the edge $e_{j}^{(i)}$ corresponding to the node $x_j^{(i)}y_j^{(i)}=t_j^{(i)}$ given by $\vert e_{j}^{(i)}\vert =\val(t_j^{(i)})$.
\end{itemize}

In this situation the special fiber of $\overline{\Div}_d^+(\calX)$ is locally given by the equations $a_i=t_1^{(i)}\cdots t_{k_i}^{(i)}$ on $\overline{\Div}_d^+(\calX)$. By Sections~\ref{section_troppolystablestandardmodel}, \ref{section_skeletonpolystablestandardmodel}, and  \ref{section_skeletonsmallneighborhood}, the retraction to the skeleton is given by sending $x\in X_d^{an}$ to $\big(-\log\vert t_1^{(i)}\vert_x,\ldots, -\log\vert t_{k_i}^{(i)}\vert_x\big)_{i=1}^r$ in $\Delta(\vec{k},\vec{a})$. But these are precisely the edge lengths $\vert e_{j}^{(i)}\vert =\val(t_j^{(i)})$ and so the above diagram commutes. 
\end{proof}

\section{Functoriality}\label{section_functoriality}

\noindent There are two classes of tautological maps associated to symmetric powers:
\begin{enumerate}[(i)]
\item For $\mu=(m_1,\ldots, m_n)\in\Z_{\geq 0}^n$ and $\delta=(d_1,\ldots, d_n)\in\Z_{\geq 0}^n$ such that $m_1d_1+\ldots +m_nd_n=d$, we have the \emph{diagonal morphism}
\begin{equation*}\begin{split}
\phi_{\mu,\delta}\colon X_{d_1}\times \cdots \times X_{d_n}&\longrightarrow X_d \\
(D_1,\ldots, D_n)& \longmapsto m_1 D_1 + \ldots m_n D_n \ .
\end{split}\end{equation*}
\item For $d\geq 0$ we have the \emph{Abel-Jacobi map}
\begin{equation*}\begin{split}
\alpha_d\colon X_d&\longrightarrow \Pic_d(X)\\
D &\longmapsto \calO_X(D) \ .
\end{split}\end{equation*} 
\end{enumerate} 
In this section, we show that the process of tropicalization naturally commutes with both classes of morphisms.

\subsection{Diagonal morphisms}
Let $\mu=(m_1,\ldots, m_n)\in\Z_{\geq 0}^n$ and $\delta=(d_1,\ldots, d_n)\in\Z_{\geq 0}^n$ be such that $m_1d_1+\ldots +m_nd_n=d$. Define the \emph{tropical diagonal map} $\phi_{\mu,\delta}^{trop}\colon \Gamma_{d_1}\times \cdots \times \Gamma_{d_n}\rightarrow \Gamma_d$  by the association
\begin{equation*}
(D_1,\ldots, D_n) \longmapsto m_1 D_1 + \ldots m_n D_n \ .
\end{equation*}

\begin{proposition}
The tropical diagonal map is a morphism of colored polysimplicial complexes that makes the diagram
\begin{center}\begin{tikzcd}
X_{d_1}^{an}\times \cdots \times X_{d_n}^{an} \arrow[rrrr,"\trop_{X_{d_1}}\times \cdots \times \trop_{X_{d_n}}"] \arrow[d,"\phi_{\mu,\delta}^{an}"] & & & & \Gamma_{d_1}\times \cdots \times \Gamma_{d_n}\arrow[d,"\phi_{\mu,\delta}^{trop}"]\\
X_d^{an}\arrow[rrrr,"\trop_{X_d}"] & & & & \Gamma_d
\end{tikzcd}\end{center}
commute. 
\end{proposition}

\begin{proof}
 The linearity of the specialization map from Section~\ref{section_specialization} implies the commutativity of 
\begin{center}\begin{tikzcd}
\Div_{K,d_1}^+(X)\times \cdots \times \Div_{K,d_n}^+(X) \arrow[rrrr,"\rho_{X,\ast}\times \cdots \times \rho_{X,\ast}"] \arrow[d,"\phi_{\mu,\delta}"] & & & & \Div_{d_1}^+(\Gamma)\times \cdots \times \Div_{d_n}^+(\Gamma)\arrow[d,"\phi_{\mu,\delta}^{trop}"]\\
\Div_{K,d}^{+}(X)\arrow[rrrr,"\rho_{X,\ast}"] & & & & \Div_d^{+}(\Gamma) \ .
\end{tikzcd}\end{center} 

Suppose first that $K=\overline{K}$ is algebraically closed. Then the monoids $\Div_{K,d_1}^+(X)\times \cdots \times \Div_{K,d_n}^+(X)$ and $\Div_{K,d}^{+}(X)$ of effective divisors are dense in $X_{d_1}^{an}\times \cdots \times X_{d_n}^{an}$ and $X_d^{an}$ respectively. Therefore the continuity of $\trop_{X_{d_1}}\times \cdots \times \trop_{X_{d_n}}$ and $\trop_{X_d}$ (coming from Theorem~\ref{thm_skeleton=symmetricpower}) together with Proposition~\ref{prop_spec=trop} implies the claim.

The general case, when $K$ may not be algebraically closed, follows from the compatibility of the tropicalization map with base changes by non-Archimedean extensions $K'$ of $K$. 
\end{proof}

\subsection{Abel-Jacobi map}
Denote by $\Rat(\Gamma)$ the abelian group of \emph{rational functions} on $\Gamma$, i.e. the group of continuous piecewise integer linear functions on $\Gamma$. There is a natural homomorphism
\begin{equation*}\begin{split}
    \div\colon \Rat(\Gamma)&\longrightarrow \Div(\Gamma)\\
    f&\longmapsto \sum_{p\in \Gamma}\ord_p(f) \cdot p
\end{split}\end{equation*}
where $\ord_p(f)$ denotes the sum of all outgoing slopes at the point $p$. Its image is the subgroup $\PDiv(\Gamma)$ of principal divisors in $\Div(\Gamma)$. One can verify that $\PDiv(\Gamma)$ is, in fact, a subgroup of $\Div_0(\Gamma)$. The \emph{Picard group} $\Pic(\Gamma)$ is defined to be the quotient $\Div(\Gamma)/\PDiv(\Gamma)$. Denote the image of a divisor $D$ on $\Gamma$ in $\Pic(\Gamma)$ by $[D]$. 

Since $\PDiv(\Gamma)$ is actually a subgroup of $\Div_0(\Gamma)$, the quotient respects degrees and $\Pic(\Gamma)$ naturally decomposes into a disjoint union of union of $\Pic_d(\Gamma)$, each of which is naturally a torsor over $\Pic_0(\Gamma)$. By the tropical Abel-Jacobi Theorem \cite[Theorem 6.2]{MikhalkinZharkov} and \cite[Theorem 3.4]{BakerFaber}, the Picard group naturally carries the structure of a principally polarized tropical abelian variety.

Let $X$ be a smooth projective curve over $K$. In \cite[Theorem 1.3]{BakerRabinoff_skelJac=Jacskel}, Baker and Rabinoff show that the non-Archimedean skeleton $\Sigma(\Pic_d(X))$ of $\Pic_d(X)^{an}$ is naturally isomorphic (as a principally polarized tropical abelian variety) to the Picard variety $\Pic_d(\Gamma)$ and that the continous retraction $\rho_{\Pic_d(X)}\colon \Pic_d(X)^{an}\rightarrow \Sigma(\Pic_d(X))$ to the skeleton naturally commutes with the tropical Abel-Jacobi map $\alpha_q\colon X\rightarrow \Pic_0(X)$ given by $p\mapsto [p-q]$ for a fixed point $q$ of $X$. We expand on their result in the following Theorem~\ref{thm_AbelJacobi}.

\begin{theorem}\label{thm_AbelJacobi}
For $d\geq 0$, the \emph{tropical Abel-Jacobi map} $\alpha_d^{trop}\colon \Gamma_d\rightarrow \Pic_d(\Gamma)$ given by the association $D\mapsto [D]$ naturally makes the diagram
\begin{center}\begin{tikzcd}
X_d^{an} \arrow[d,"\alpha_d^{an}"] \arrow[rrr,"\trop_{X_d}"] &&& \Gamma_d \arrow[d,"\alpha_d^{trop}"] \\
\Pic_d(X)^{an} \arrow[rr,"\rho_{\Pic_d(X)}"] && \Sigma(\Pic_d(X))\arrow[r,"\sim"]&\Pic_d(\Gamma)
\end{tikzcd}\end{center}
commute. 
\end{theorem}

We remark that a version of Theorem \ref{thm_AbelJacobi} has also appeared in \cite[Section 7]{Shen_Lefschetz}. We include the proof here for completeness.

\begin{proof}[Proof of Theorem \ref{thm_AbelJacobi}]
Suppose first that $K=\overline{K}$ is algebraically closed. There is a natural homomorphism 
\begin{equation*}
    \trop\colon \Rat(X)^\ast\longrightarrow \Rat(\Gamma)
\end{equation*} that is given by sending a non-zero rational function $f\in\Rat(X)^\ast$ to the map $x\mapsto -\log\vert f\vert_x$ on $\Gamma$, thought of as the non-Archimedean skeleton of $X^{an}$. Since $K$ is algebraically closed, we have $\Div(X)=\Div_K(X)$. By the slope formula \cite[Theorem 5.14]{BakerPayneRabinoff}, we have $\div(\trop(f))=\rho_{X,\ast}(\div(f))$ where $\rho_{X,\ast}$ denotes the specialization map $\rho_{X,\ast}\colon\Div(X)\simeq\Div_K(X)\rightarrow \Div(\Gamma)$ discussed in Section~\ref{section_specialization} above. Therefore the specialization map descends to a homomorphism
\begin{equation*}
    \rho_{X,\ast}\colon \Pic_d(X)\longrightarrow \Pic_d(\Gamma) 
\end{equation*}
and this immediately implies that the diagram
\begin{equation}\label{eq_AbelJacobikbar}
    \begin{tikzcd}
        \Div_d^+(X)\arrow[rr,"\rho_{X,\ast}"]\arrow[d,"\alpha_d"]&&\Div_d^+(\Gamma) \arrow[d,"\alpha_d^{trop}"]\\
        \Pic_d(X)\arrow[rr,"\rho_{X,\ast}"]&&\Pic_d(\Gamma)
    \end{tikzcd}
\end{equation}
commutes.

In \cite[Proposition 5.3]{BakerRabinoff_skelJac=Jacskel}, the authors show that the Picard group $\Pic_d(\Gamma)$ is naturally isomorphic (as a principally polarized tropical abelian variety) to the non-Archimedean skeleton $\Sigma(\Pic_d(X))$ of $\Pic_d(X)^{an}$ such that the induced diagram
\begin{center}
    \begin{tikzcd}
       \Pic_d(X) \arrow[rr,"\rho_{X,\ast}"] 
       \arrow[d,"\subseteq"]&& \Pic_d(\Gamma)\arrow[d,"\simeq"]\\
       \Pic_d(X)^{an}\arrow[rr,"\rho_{\Pic_d(X)}"]&& \Sigma(\Pic_d(X))
    \end{tikzcd}
\end{center}
commutes. In fact, Baker and Rabinoff only show this statement for $\Pic_0(X)$, but since $K$ is algebraically closed, we may choose a point $p\in X(K)$ and identify $\Pic_d(X)$ with $\Pic_0(X)$. 

Note that, since $K$ is assumed to be algebraically closed, both $\Div_d^+(X)$ and $\Pic_d(X)$ are dense in $X_d^{an}$ and $\Pic_d(X)^{an}$.  Therefore, by Propostion~\ref{prop_spec=trop} above, since the maps $\rho_{\overline{\Div}_d^+(\calX)}$, $\rho_{\Pic_d(X)}$, and $\alpha_d$ are all continuous, the commutativity of diagram \eqref{eq_AbelJacobikbar} implies that
\begin{center}\begin{tikzcd}
X_d^{an} \arrow[d,"\alpha_d^{an}"] \arrow[rrr,"\trop_{X_d}"] &&& \Gamma_d \arrow[d,"\alpha_d^{trop}"] \\
\Pic_d(X)^{an} \arrow[rr,"\rho_{\Pic_d(X)}"] && \Sigma(\Pic_d(X))\arrow[r,"\sim"]&\Pic_d(\Gamma)
\end{tikzcd}\end{center}
commutes. 

The general case, when $K$ may not be algebraically closed, again follows from the invariance of the projection to the skeleton under base change by non-Archimedean field extensions.
\end{proof}

 


 


\section{A Bieri-Groves-Theorem}\label{section_BieriGroves}

In this section we deduce the following Theorem~\ref{thm_BieriGroves} from the Bieri-Groves-Theorem (see \cite[Theorem A]{BieriGroves} and \cite[Theorem 2.2.3]{EinsiedlerKapranovLind}), which immediately implies Theorem~\ref{thm_BieriGrovessympow} from the introduction. 

\begin{theorem}\label{thm_BieriGroves}
Let $X$ be a smooth and proper variety over $K$ and let $\calX$ be a proper polystable model of $X$. Suppose that $Y$ is a closed subvariety of $X$ (defined over $K$) that is equidimensional of dimension $\delta$. Then the tropicalization 
\begin{equation*}
    \Trop_{\calX}(Y):=\rho_{\calX}(Y^{an})\subseteq \Sigma(\calX)
\end{equation*}
of $Y$ (as a subspace of $X$) is a $\Lambda$-rational polyhedral complex in $\Sigma(\calX)$ of dimension $\leq \delta$. If $\calX$ has a deepest stratum $E$ that is a point and $\Trop_\calX(Y)$ contains a point in the interior of $\Sigma_E$ (as in Section \ref{section_skeletonsmallneighborhood}), then the dimension of $\Trop_\calX(Y)$ is equal to $\delta$. 
\end{theorem}

Let $\G_m^n=\Spec K\big[t_1^{\pm 1},\ldots, t_n^{\pm 1}\big]$ be a split algebraic torus over $K$. Recall e.g. from \cite{Gubler_guide} that there is a natural proper and continuous tropicalization map
\begin{equation*}
    \trop_{\G_m^n}\colon \big(\G_m^n\big)^{an}\longrightarrow \R^n \ .
\end{equation*}
It is given by sending a point $x\in\big(\G_m^n\big)^{an}$, which corresponds to a multiplicative seminorm $\vert .\vert_x$ on $K\big[t_1^{\pm 1},\ldots, t_n^{\pm 1}\big]$ extending the absolute value on $K$, to the point $\big(\vert t_1\vert_x,\ldots, \vert t_n\vert_x\big)\in\R^n$. Using this map, the \emph{tropicalization} of a subvariety $Y\subseteq \G_m^n$ may be defined to be the projection 
\begin{equation*}
    \Trop_{\G_m^n}(Y):=\trop_{\G_m^n}(Y^{an})
\end{equation*}
of $Y^{an}\subseteq \big(\G_m^n\big)^{an}$ to $\R^n$. 

\begin{lemma}\label{lemma_localtrop=actualtrop}
Let $\calZ_{\vec{n},\vec{k},\vec{a}}$ be a standard polystable model over $R$. Then the generic fiber is the algebraic torus $\G_m^{\vert \vec{n} \vert}$ and the natural diagram
\begin{center}
    \begin{tikzcd}
        \calZ_{\vec{n},\vec{k},\vec{a}}^\circ \arrow[rr,"\trop_{\vec{n},\vec{k},\vec{a}}"]\arrow[d,hookrightarrow,"\subseteq"'] && \Delta(\vec{k},\vec{a}) \arrow[d,hookrightarrow,"\subseteq"]\\
        \big(\G_m^{\vert \vec{n}\vert}\big)^{an} \arrow[rr,"\trop_{\G_m^{\vert\vec{n}\vert}}"]&&\R^{\vert\vec{n}\vert}
    \end{tikzcd}
\end{center}
of tropicalization maps commutes. 
\end{lemma}

\begin{proof} Recall from Section \ref{section_polystableskeletons} that a \emph{polystable standard model} is an affine $R$-scheme of the form 
\begin{equation*}
    \calZ_{\vec{n},\vec{k},\vec{a}}=\calZ(n_1,k_1,a_1)\times_R\cdots\times_R\calZ(n_r,k_r,a_r) 
\end{equation*}
where each $\calZ(n_i,k_i,a_i)=\Spec A_i$ with $A_i=R[t^{(i)}_0,\ldots, t^{(i)}_{n_i}]/(t^{(i)}_0\cdots t^{(i)}_{k_i}-a_i)$ for $a_i\in R$ and  $\vec{n}=(n_1,\ldots, n_r)\in\NN^r$, $\vec{k}=(k_1,\ldots, k_r)\in\NN^r$ with $k_i\leq n_i$ and $\vec{a}=(a_1,\ldots, a_r)\in R^r$. Notice that the scheme-theoretic generic fiber is isomorphic to $\mathbb{G}_m^{\vert \vec{n}\vert}$, where $\vert \vec{n}\vert=n_1+\cdots +n_r$ and so $\calZ_{\vec{n},\vec{k},\vec{a}}^\circ$ is naturally a subset of $(\G_m^{\vert n\vert})^{an}$.

Given a point $x\in\calZ_{\vec{n},\vec{k},\vec{a}}^\circ$, using the definition from Section \ref{section_troppolystablestandardmodel}, we have: 
\begin{equation*}
    \trop_{\vec{n},\vec{k},\vec{a}}(x)=\Big(-\log\big\vert t_1^{(1)}\big\vert_x, \ldots, -\log\big\vert t_{k_1}^{(1)}\big\vert_x, \ldots, -\log\big\vert t_{1}^{(r)}\big\vert,\ldots, -\log\big\vert t_{k_r}^{(r)}\big\vert_x\Big) \ .
\end{equation*}
Let $j=k_i+1,\ldots, n_i$ and $i=1,\ldots, r$. Since $\vert .\vert_x$ is multiplicative, we have
\begin{equation*}
    \big\vert (t^{(i)}_j)\big\vert_x \cdot\big\vert (t^{(i)}_j)^{-1}\big\vert_x=\big\vert 1\big\vert_x=1 \ .
\end{equation*} 
Since $\vert.\vert_x$ is bounded and both $t_i^{(j)}$ and $(t_i^{(j)})^{-1}$ pull back to the coordinate ring of $\calZ_{\vec{n},\vec{k},\vec{a}}^\circ$, 
we have $\big\vert t^{(i)}_{j}\big\vert_x\leq 1$ as well as $\big\vert (t^{(i)}_{j})^{-1}\big\vert_x\leq 1$. This implies $\big\vert t_{j}^{(i)}\big\vert_x=1$ for $j=k_i+1,\ldots, n_i$ and $i=1,\ldots, r$. Therefore we have 
\begin{equation*}\begin{split}
    \trop_{\G_m^{\vert \vec{n}\vert}}(x)=&
    \big(-\log\big\vert t_1^{(1)}\big\vert_x,\ldots, -\log\big\vert t_{n_1}^{(1)}\big\vert_x,\ldots, -\log\big\vert t_{1}^{(r)}\big\vert_x, \ldots,-\log\big\vert t_{n_r}^{(r)}\big\vert_x \big)\\
    =&\Big(-\log\big\vert t_1^{(1)}\big\vert, \ldots, -\log\big\vert t_{k_1}^{(1)}\big\vert_x,0,\ldots, 0,\\
     &-\log\big\vert t_{1}^{(2)}\big\vert_x, \ldots, -\log\big\vert t_{k_2}^{(2)}\big\vert_x, 0, \ldots,0,\\
     & \ldots \\
    &  -\log\big\vert t_{1}^{(r-1)}\big\vert,\ldots, -\log\big\vert t_{k_{r-1}}^{(r-1)}\big\vert_x,0, \ldots,0,\\ 
    &  -\log\big\vert t_{1}^{(r)}\big\vert,\ldots, -\log\big\vert t_{k_r}^{(r)}\big\vert_x,0, \ldots,0 \Big)
\end{split}\end{equation*}
which is precisely the image of $\trop_{\vec{n},\vec{k},\vec{a}}(x)$ under the embedding $\R^{k_1+\ldots +k_r}\hookrightarrow \R^{n_1+\ldots+n_r}$.
\end{proof}

The proof of Theorem~\ref{thm_BieriGroves} closely follows along the lines of the proof of \cite[Theorem 1.1]{Ulirsch_tropcomplogreg}, the Bieri-Groves-Theorem for subspaces of log-regular varieties. 

\begin{proof}[Proof of Theorem~\ref{thm_BieriGroves}]
We need to show that that $\Trop_\calU(Y)=\Trop_\calX(Y)\cap\Sigma(\calU)$ is a $\Lambda$-rational polyhedral complex for a small 
\'etale open subset $\calU$ around every stratum $E$ of $\calX_0$. We may choose $\calU$ so that there is a small chart $\gamma\colon \calU\rightarrow \calZ(\vec{n},\vec{k},\vec{a})$. By the local description of $\rho_\calU$ in terms of the tropicalization map $\trop_{\vec{n},\vec{k},\vec{a}}\colon \calZ(\vec{n},\vec{k},\vec{a})^\circ\rightarrow \Delta(\vec{k},\vec{a})$ in Section~\ref{section_polystableskeletons}, we may identify $\Trop_\calU(Y)$ with the projection $\trop_{\vec{n},\vec{k},\vec{a}}(\gamma^\circ(Y^{an}\cap\calU^\circ))$ and, by Lemma~\ref{lemma_localtrop=actualtrop}, with $\trop_{\G_m^{\vert \vec{n}\vert}}(\gamma^\circ(Y^{an}\cap\calU^\circ))$.

Since $\gamma$ is 
\'etale, the image $\gamma(Y\cap U)$ is locally closed in $\G_m^{\vert\vec{n}\vert}$. Denote by $\overline{Y}_\gamma$ its closure in $\G_m^{\vert\vec{n}\vert}$. By a generalization of Draisma's tropical lifting lemma \cite[Lemma 4.4]{Draisma_tropicalsecants} (see \cite[Proposition 11.5]{Gubler_guide} and \cite[Lemma 3.10]{Ulirsch_tropcomplogreg}), the tropicalization of a locally closed subset is equal to the tropicalization of its closure and so we have
\begin{equation*}
    \Trop_\calU(Y^{an})=\Trop(\overline{Y}_\gamma)\cap\Delta(\vec{k},\vec{a}) \ .
\end{equation*}

The tropicalization $\Trop_\calU(Y^{an})$ is a $\Lambda$-rational polyhedral complex of dimension $\delta$ by the classical Bieri-Groves-Theorem \cite[Theorem A]{BieriGroves} and \cite[Theorem 2.2.3]{EinsiedlerKapranovLind}. Due to the intersection with $\Delta(\vec{k},\vec{a})$, the tropicalization $\Trop_\calU(Y^{an})$ might have dimension $\leq \delta$. If $E$ is a point, then the cell $\Sigma_E$ will be $\vert\vec{n}\vert$-dimensional and, if $\Trop_\calU(Y^{an})$ has a point in the interior of a cell, it will be part of an $\delta$-dimensional cell of $\Trop(\overline{Y}_\gamma)$ whose intersection with $\Delta(\vec{k},\vec{a})=\Sigma_E$ is $\delta$-dimensional, since $\Trop_{\G_m^{\vert\vec{n}\vert}}(\overline{Y}_\gamma)$ fulfills the balancing condition. 
\end{proof}


\section{Faithful tropicalization of polystable skeletons}\label{section_faithfultropicalization}

In this section we prove the following Theorem~\ref{thm_faithfultropicalization}, which is a generalization of \cite[Theorem 9.5]{GublerRabinoffWerner} to the  polystable case (when no extra divisor at infinity is present). It immediately implies Theorem~\ref{thm_faithfulsympow} from the introduction.

\begin{theorem}\label{thm_faithfultropicalization}
Let $X$ be a smooth and proper variety over $K$ and let $\calX$ be a proper  polystable model of $X$ over $R$. Then there is an open subset $U\subseteq X$ as well as a morphism $f\colon U\rightarrow \G_m^n$ such that the restriction of
$\trop\circ f^{an}$ to  $\Sigma(\calX)\subseteq U^{an}$ of
the induced tropicalization map
\begin{equation*}
\trop_f\colon U^{an}\xlongrightarrow{f^{an}}\G_m^{n,an}\xlongrightarrow{\trop}\R^n
\end{equation*}
is a homeomorphism onto its image in $\mathbb{R}^n$.
If $\mathcal{X}$ is a strictly polystable model for $X$, then the restriction of $\trop\circ f^{an}$ to each cell of $\Sigma(\calX)$ is unimodular.
\end{theorem}

\begin{proof}
Suppose that $\calX$ is strictly polystable. Fix a stratum $E$ of $\calX_0$ and choose a small chart $\gamma\colon \calU\rightarrow \calZ(\vec{n},\vec{k},\vec{a})$ around $E$ (with $\calU\subseteq \calX$). Let $U_{(E)}:=\calU_K$ be the generic fiber and let  $f_{(E)}\colon U_{(E)}\rightarrow \G_m^{\vert\vec{n}
\vert}$ be the base change of $\gamma$ to the generic fiber. By Lemma~\ref{lemma_localtrop=actualtrop} the tropicalization map 
\begin{equation*}
    \trop_{f_{(E)}}\colon U_{(E)}^{an}\xlongrightarrow{f_{(E)}^{an}} \G_m^{\vert \vec{n}\vert,an}\xlongrightarrow{\trop}\R^{\vert \vec{n}\vert}
\end{equation*}
naturally restricts to the projection to $\Delta(\vec{n},\vec{k},\vec{a})$  on $\calU^\circ\subseteq U_{(E)}^{an}$. Therefore, by Section~\ref{section_skeletonsmallneighborhood}, the restriction to the skeleton $\Sigma(\calU)\subseteq \calU^\circ\subseteq U_{(E)}^{an}$ is a unimodular homeomorphism onto its image in $\R^{\vert\vec{n}\vert}$.

Set $U=\bigcap_{E}U_{(E)}$, where $E$ is passing through all the strata of $\calX_0$. We take ${f=(f_{(E)})\colon U\longrightarrow \G_m^n}$, where $n=\sum_{E}\vert\vec{n}_e\vert$ and $f$ is given by  $f(x)= \big(f_{(E)}(x)\big)_{E}\in\G_m^{n}$. The above reasoning shows that the restriction to the skeleton $\Sigma(\calX)$ is a  homeomorphism on every polysimplex in $\Sigma(\calX)$, and in the case that $\calX$ is strictly polystable, it is unimodular as well. 

We now show that $\trop_f$ is injective: Consider two points $x,x'\in \Sigma(\calX)$ such that $\trop_f(x)=\trop_f(x')$. Then $\trop_{f_{(E)}}(x)=\trop_{f_{(E)}}(x')$ for all strata $E$ of $\calX_0$.
Suppose that $x$ is in the relative interior of $\Delta(E)$ and $x'$ is in the relative interior of $\Delta(E')$. By Lemma~\ref{lemma_localtrop=actualtrop} and Section~\ref{section_skeletonsmallneighborhood}, we then have
\begin{equation*}
    x=\rho_\calX(x)=\rho_\calX(x') \ .
\end{equation*}
Since $x$ is in the relative interior of $\Delta(E)$, the point $x'$ is in the relative interior of $\Delta(E)$ as well, by the construction of $\Sigma(\calX)$ as a geometric realization of the polysimplicial complex associated to $\calX$ (see Proposition~\ref{prop_skel=sympow}). In particular, we have $E=E'$ and $x=x'$, since the restriction of $\trop_{f_{(E)}}$ to $\Delta(E)\subseteq\Sigma(\calX)$ is injective. 

In general, if $\calX$ is a only polystable model, we may apply a barycentric subdivision to $\Sigma(\calX)$ that corresponds to a toroidal modification of $\calX$ over $R$ making it strictly polystable. The argument in the strictly polystable case now yields a unimodular faithful tropicalization of the barycentric subdivision. Since the barycentric subdivision of $\Sigma(\calX)$ is homeomorphic to $\Sigma(\calX)$ we find that that $\Sigma(\calX)$ maps homeomorphically onto its image in $\R^n$.
\end{proof}

\section{Open questions}

\subsection{Effective faithful tropicalization via linear series}\label{section_faithfultropvialinearseries}
Let $L$ be a line bundle. For $g\geq 0$ set
\begin{equation*}
    t(g)=\begin{cases} 1 & \textrm{ if } g=0 \\
    3 &\textrm{ if } g=1\\
    3g-1 &\textrm{ if } g\geq 2 \ .
    \end{cases}
\end{equation*}

In \cite{KawaguchiYamaki}, Kawaguchi and Yamaki show that, if $\deg L\geq t(g)$, then there are sections $s_0,\ldots, s_r\in H^0(X,L)$ such that the associated map 
\begin{equation*}
    \begin{split}
        X&\longrightarrow \PP^r\\
        x&\longmapsto \big(s_0(x),\ldots, s_r(x)\big)
    \end{split}
\end{equation*}
induces a tropicalization map
\begin{equation*}\begin{split}
    \trop_{(s_0,\ldots,s_r)}\colon X^{an}&\longrightarrow \TT\PP^r\\
    x&\longmapsto \big(-\log\vert s_0\vert_x,\ldots, -\log\vert s_r\vert_x\big)
\end{split}\end{equation*}
that is faithful on the skeleton $\Gamma$ of $X$, i.e. that restricts to a piecewise integer linear and unimodular map on $\Gamma$.

Let $L$ be an $(d-1)$-ample line bundle on $X$. There is a natural map 
\begin{equation*}
X_d\longrightarrow \Gr\big(d,H^0(X,L)^\ast\big)
\end{equation*}
into the Grassmannian of $d$-dimensional quotients of $H^0(X,L)$ that is given by associating to an effective divisor $D$ on $X$ the surjective restriction map
\begin{equation*}
H^0(X,L)\twoheadrightarrow H^0(X,L\otimes \calO_D) \ .
\end{equation*}
If $L$ is $d$-ample, this map is an injection. 

Choosing global sections $s_0,\ldots, s_r\in H^0(X,L)$ we find a map 
\begin{equation*}
    X_d\longrightarrow \Gr(d,r)
\end{equation*}
that on the open locus $X_d^\circ$ (parametrizing reduced divisors on $X$) is given by sending a split reduced effective divisor $D=\sum_{i=1}^d p_i$ to the linear space spanned by the $p_i$. If we compose this with the Pl\"ucker embedding, we obtain a map $X_d\rightarrow \PP^N$ with $N=\binom{r}{k}-1$ such that the vanishing of the Pl\"ucker coordinates precisely describes the locus of non-reduced divisors. In other words, we have $X_d\cap\G_m^N=X_d^\circ$. 

Expanding on the work of Kawaguchi and Yamaki \cite{KawaguchiYamaki}, one might be tempted to ask the following:

\begin{question}
Suppose that $L$ is a line bundle on $X$ that is $d$-ample.
Under which conditions is there a basis $s_0,\ldots, s_r$ of $H^0(X,L)$ such that the induced tropicalization map 
\begin{equation*}
X_d^{\circ, an}\hooklongrightarrow \Gr^\circ\big(d,r\big)^{an}\hooklongrightarrow \G_m^{N,an}\xlongrightarrow{\trop_{\G_m^{N}}}\R^N
\end{equation*}
is faithful on the skeleton $\Sigma(X_d)\simeq \Gamma_d$? 
\end{question}

One can think of the desired condition as a tropical analogue of $d$-ampleness.

\subsection{de Jonqui\'{e}res divisors}\label{section_deJonquieres}

We now discuss de Jonqui\'{e}res divisors, and we begin with a discussion of the classical case.
Let $X$ be a smooth projective curve with genus $g$.
Consider a fixed complete linear series $l = (L,V)$ of degree $d$ and dimension $r$. Then a \emph{de Jonqui\'{e}res divisor of length $N$} is a divisor $a_1 D_1 + \cdots + a_k D_k \in X_d$ contained in $\mathbb{P}V$ that fulfills $\sum_{i=1}^k \deg(D_i) = N$. These are studied extensively in \cite{Ungureanu_deJonquieres}. If $\mu_1 = (a_1, \ldots, a_k)$ and $\mu_2=(d_1, \ldots, d_k)$ are positive partitions such that $\sum_{i=1}^k a_i d_i = d$, then we denote the set of de Jonqui\'{e}res divisors of length $N$ determined by $\mu_1$ and $\mu_2$ by $DJ_{k,N}^{r,d}(\mu_1,\mu_2,C,l)$. In \cite{Ungureanu_deJonquieres}, the author proves that for general curves, if $N-d+r \geq 0$, then $DJ_{k,N}^{r,d}(\mu_1,\mu_2,C,l)$ has the expected dimension $N-d+r$. In particular, when $N - d + r < 0$, the variety $DJ_{k,N}^{r,d}(\mu_1,\mu_2,C,l)$ is empty. 

One may wonder whether this result remains true tropically. The following proposition addresses the emptiness result and would imply its algebraic counterpart.


\begin{proposition}
Let $\Gamma$ be a generic chain of loops and $K$ its canonical divisor, so $d = 2g-2$ and $r = g-1$. If $n$ is such that $n + d -r < 0$, then $|K|$ does not contain a divisor of the form $d_1 p_1 + \cdots d_n p_n$. 
\end{proposition}
\begin{proof}
The canonical divisor has degree $2g-2$ and rank $g-1$. Therefore, in order for $n + d -r < 0$ to hold, we must have $n < g-1$. The canonical divisor is supported on $g-1$ vertices, and because of the genericity condition, any divisor equivalent to the canonical divisor will have at least that many vertices in its support. Therefore, there is no divisor of the form $d_1 p_1 + \cdots d_n p_n$ when $n<g-1$. 
\end{proof}

However, unlike in the classical case, the result does not hold for all divisors. We give the following example.

\begin{example}
Consider the length 2 generic chain of loops, and let $p$ be the middle vertex. Let $D = K + p$. Then the rank of $D$ is 1 (because there is a divisor $p_1 + p_2$, with each point coming from a separate loop, such that $D-D'$ is not effective), and so if $n = 1$ then $n-d+1 = -1<0$. So, in the classical case we would expect there to be no divisor in $|D|$ of the form $3q$ for $q \in \Gamma$. However, in this case $D = 3p$.
\end{example}


\bibliographystyle{amsalpha}
\bibliography{biblio}{}

\providecommand{\bysame}{\leavevmode\hbox to3em{\hrulefill}\thinspace}
\providecommand{\MR}{\relax\ifhmode\unskip\space\fi MR }
\providecommand{\MRhref}[2]{%
  \href{http://www.ams.org/mathscinet-getitem?mr=#1}{#2}
}
\providecommand{\href}[2]{#2}
\begin{thebibliography}{CCUW20}

\bibitem[AC13]{AminiCaporaso}
Omid Amini and Lucia Caporaso, \emph{Riemann-{R}och theory for weighted graphs
  and tropical curves}, Adv. Math. \textbf{240} (2013), 1--23.

\bibitem[ACP15]{AbramovichCaporasoPayne}
Dan Abramovich, Lucia Caporaso, and Sam Payne, \emph{The tropicalization of the
  moduli space of curves}, Ann. Sci. \'{E}c. Norm. Sup\'{e}r. (4) \textbf{48}
  (2015), no.~4, 765--809.

\bibitem[AOV11]{AbramovichOlssonVistoli}
Dan Abramovich, Martin Olsson, and Angelo Vistoli, \emph{Twisted stable maps to
  tame {A}rtin stacks}, J. Algebraic Geom. \textbf{20} (2011), no.~3, 399--477.

\bibitem[AP20]{AbreuPacini}
Alex Abreu and Marco Pacini, \emph{The universal tropical {J}acobian and the
  skeleton of the {E}steves' universal {J}acobian}, Proc. Lond. Math. Soc. (3)
  \textbf{120} (2020), no.~3, 328--369.

\bibitem[Bak08]{Baker_specialization}
Matthew Baker, \emph{Specialization of linear systems from curves to graphs},
  Algebra Number Theory \textbf{2} (2008), no.~6, 613--653, With an appendix by
  Brian Conrad.

\bibitem[Ber90]{Berkovich_book}
Vladimir~G. Berkovich, \emph{Spectral theory and analytic geometry over
  non-{A}rchimedean fields}, Mathematical Surveys and Monographs, vol.~33,
  American Mathematical Society, Providence, RI, 1990.

\bibitem[Ber99]{Berkovich_analytic}
\bysame, \emph{Smooth {$p$}-adic analytic spaces are locally contractible},
  Invent. Math. \textbf{137} (1999), no.~1, 1--84.

\bibitem[BF11]{BakerFaber}
Matthew Baker and Xander Faber, \emph{Metric properties of the tropical
  {A}bel-{J}acobi map}, J. Algebraic Combin. \textbf{33} (2011), no.~3,
  349--381.

\bibitem[BG84]{BieriGroves}
Robert Bieri and J.~R.~J. Groves, \emph{The geometry of the set of characters
  induced by valuations}, J. Reine Angew. Math. \textbf{347} (1984), 168--195.

\bibitem[BM19]{BrownMazzon}
Morgan~V. Brown and Enrica Mazzon, \emph{The essential skeleton of a product of
  degenerations}, Compos. Math. \textbf{155} (2019), no.~7, 1259--1300.

\bibitem[BN07]{BakerNorine}
Matthew Baker and Serguei Norine, \emph{Riemann-{R}och and {A}bel-{J}acobi
  theory on a finite graph}, Adv. Math. \textbf{215} (2007), no.~2, 766--788.

\bibitem[BPR13]{BakerPayneRabinoff_structure}
Matthew Baker, Sam Payne, and Joseph Rabinoff, \emph{On the structure of
  non-{A}rchimedean analytic curves}, Tropical and non-{A}rchimedean geometry,
  Contemp. Math., vol. 605, Amer. Math. Soc., Providence, RI, 2013,
  pp.~93--121.

\bibitem[BPR16]{BakerPayneRabinoff}
\bysame, \emph{Nonarchimedean geometry, tropicalization, and metrics on
  curves}, Algebr. Geom. \textbf{3} (2016), no.~1, 63--105.

\bibitem[BR15]{BakerRabinoff_skelJac=Jacskel}
Matthew Baker and Joseph Rabinoff, \emph{The skeleton of the {J}acobian, the
  {J}acobian of the skeleton, and lifting meromorphic functions from tropical
  to algebraic curves}, Int. Math. Res. Not. IMRN (2015), no.~16, 7436--7472.

\bibitem[CCUW20]{CavalieriChanUlirschWise_tropstack}
Renzo Cavalieri, Melody Chan, Martin Ulirsch, and Jonathan Wise, \emph{A moduli
  stack of tropical curves}, Forum of Mathematics, Sigma \textbf{8} (2020)
  (en).

\bibitem[CDPR12]{CoolsDraismaPayneRobeva}
Filip Cools, Jan Draisma, Sam Payne, and Elina Robeva, \emph{A tropical proof
  of the {B}rill-{N}oether theorem}, Adv. Math. \textbf{230} (2012), no.~2,
  759--776.

\bibitem[CGP18]{ChanGalatiusPayne}
Melody Chan, Soren Galatius, and Sam Payne, \emph{Tropical curves, graph
  homology, and top weight cohomology of {M}\_g}, arXiv:1805.10186 [math]
  (2018).

\bibitem[CHMR16]{CavalieriHampeMarkwigRanganathan}
Renzo Cavalieri, Simon Hampe, Hannah Markwig, and Dhruv Ranganathan,
  \emph{Moduli spaces of rational weighted stable curves and tropical
  geometry}, Forum Math. Sigma \textbf{4} (2016), e9, 35.

\bibitem[CHW14]{CuetoHaebichWerner}
Maria~Angelica Cueto, Mathias H\"{a}bich, and Annette Werner, \emph{Faithful
  tropicalization of the {G}rassmannian of planes}, Math. Ann. \textbf{360}
  (2014), no.~1-2, 391--437.

\bibitem[CJP15]{CartwrightJensenPayne}
Dustin Cartwright, David Jensen, and Sam Payne, \emph{Lifting divisors on a
  generic chain of loops}, Canad. Math. Bull. \textbf{58} (2015), no.~2,
  250--262.

\bibitem[CMR16]{CavalieriMarkwigRanganathan}
Renzo Cavalieri, Hannah Markwig, and Dhruv Ranganathan, \emph{Tropicalizing the
  space of admissible covers}, Math. Ann. \textbf{364} (2016), no.~3-4,
  1275--1313.

\bibitem[Dra08]{Draisma_tropicalsecants}
Jan Draisma, \emph{A tropical approach to secant dimensions}, J. Pure Appl.
  Algebra \textbf{212} (2008), no.~2, 349--363.

\bibitem[EKL06]{EinsiedlerKapranovLind}
Manfred Einsiedler, Mikhail Kapranov, and Douglas Lind, \emph{Non-{A}rchimedean
  amoebas and tropical varieties}, J. Reine Angew. Math. \textbf{601} (2006),
  139--157.

\bibitem[GK08]{GathmannKerber}
Andreas Gathmann and Michael Kerber, \emph{A {R}iemann-{R}och theorem in
  tropical geometry}, Math. Z. \textbf{259} (2008), no.~1, 217--230.

\bibitem[GRW16]{GublerRabinoffWerner}
Walter Gubler, Joseph Rabinoff, and Annette Werner, \emph{Skeletons and
  tropicalizations}, Adv. Math. \textbf{294} (2016), 150--215.

\bibitem[GRW17]{GublerRabinoffWernerII}
\bysame, \emph{Tropical skeletons}, Ann. Inst. Fourier (Grenoble) \textbf{67}
  (2017), no.~5, 1905--1961.

\bibitem[Gub13]{Gubler_guide}
Walter Gubler, \emph{A guide to tropicalizations}, Algebraic and combinatorial
  aspects of tropical geometry, Contemp. Math., vol. 589, Amer. Math. Soc.,
  Providence, RI, 2013, pp.~125--189.

\bibitem[Has03]{Hassett}
Brendan Hassett, \emph{Moduli spaces of weighted pointed stable curves}, Adv.
  Math. \textbf{173} (2003), no.~2, 316--352.

\bibitem[HMY12]{HaaseMusikerYu}
Christian Haase, Gregg Musiker, and Josephine Yu, \emph{Linear systems on
  tropical curves}, Math. Z. \textbf{270} (2012), no.~3-4, 1111--1140.

\bibitem[KY21]{KawaguchiYamaki}
Shu Kawaguchi and Kazuhiko Yamaki, \emph{Effective faithful tropicalizations
  associated to linear systems on curves}, Mem. Amer. Math. Soc. \textbf{270}
  (2021), no.~1323, v+110.

\bibitem[Mil86]{Milne_Jacobian}
J.~S. Milne, \emph{Jacobian varieties}, Arithmetic geometry ({S}torrs, {C}onn.,
  1984), Springer, New York, 1986, pp.~167--212.

\bibitem[MOP11]{MarianOpreaPandharipande}
Alina Marian, Dragos Oprea, and Rahul Pandharipande, \emph{The moduli space of
  stable quotients}, Geom. Topol. \textbf{15} (2011), no.~3, 1651--1706.

\bibitem[MUW21]{MoellerUlirschWerner_realizability}
Martin M\"{o}ller, Martin Ulirsch, and Annette Werner, \emph{Realizability of
  tropical canonical divisors}, J. Eur. Math. Soc. (JEMS) \textbf{23} (2021),
  no.~1, 185--217.

\bibitem[MZ08]{MikhalkinZharkov}
Grigory Mikhalkin and Ilia Zharkov, \emph{Tropical curves, their {J}acobians
  and theta functions}, Curves and abelian varieties, Contemp. Math., vol. 465,
  Amer. Math. Soc., Providence, RI, 2008, pp.~203--230.

\bibitem[Ran17]{Ranganathan_skeletonsofstablemaps}
Dhruv Ranganathan, \emph{Skeletons of stable maps {I}: rational curves in toric
  varieties}, J. Lond. Math. Soc. (2) \textbf{95} (2017), no.~3, 804--832.

\bibitem[RSPW19]{RanganathanSantosParkerWiseII}
Dhruv Ranganathan, Keli Santos-Parker, and Jonathan Wise, \emph{Moduli of
  stable maps in genus one and logarithmic geometry, {II}}, Algebra Number
  Theory \textbf{13} (2019), no.~8, 1765--1805.

\bibitem[She16]{Shen_Lefschetz}
Tif Shen, \emph{A {Lefschetz} {Hyperplane} {Theorem} for non-{Archimedean}
  {Jacobians}}, arXiv:1610.02417 [math] (2016).

\bibitem[Thu07]{Thuillier_toroidal}
Amaury Thuillier, \emph{G\'{e}om\'{e}trie toro\"{i}dale et g\'{e}om\'{e}trie
  analytique non archim\'{e}dienne. {A}pplication au type d'homotopie de
  certains sch\'{e}mas formels}, Manuscripta Math. \textbf{123} (2007), no.~4,
  381--451.

\bibitem[Uli15a]{Ulirsch_tropcomplogreg}
Martin Ulirsch, \emph{Tropical compactification in log-regular varieties},
  Math. Z. \textbf{280} (2015), no.~1-2, 195--210.

\bibitem[Uli15b]{Ulirsch_tropHassett}
\bysame, \emph{Tropical geometry of moduli spaces of weighted stable curves},
  J. Lond. Math. Soc. (2) \textbf{92} (2015), no.~2, 427--450.

\bibitem[Ung16]{Ungureanu_deJonquieres}
Mara Ungureanu, \emph{Dimension theory and degenerations of de {J}onqui{\`e}res
  divisors}, arXiv:1612.03141 [math] (2016).

\end{thebibliography}

\end{document}